\RequirePackage{fix-cm}
\documentclass[smallcondensed]{svjour3}     %
\smartqed  %

\usepackage[square,numbers]{natbib}

\usepackage[utf8]{inputenc}
\usepackage[english]{babel}
\usepackage{graphicx}
\usepackage{subcaption}
\usepackage{mathptmx}      %
\usepackage[short]{optidef}
\usepackage{amsmath, amssymb, amsfonts, mathtools, amsthm}
\usepackage{algorithm,algpseudocode}
\usepackage{xcolor}
\usepackage{booktabs}
\usepackage{multirow}
\usepackage{dsfont}
\usepackage[most]{tcolorbox}
\usepackage{eucal, comment}
\usepackage[hidelinks]{hyperref}
\usepackage{longtable}
\usepackage{siunitx}    %
\usepackage{pdflscape}

\newtheorem{def_private}{Definition}
\newtheorem{assumption}{Assumption}
\newcommand{\ra}[1]{\renewcommand{\arraystretch}{#1}}
\newcommand{\andrea}[1]{#1}

\usepackage{acronym}

\newacro{DM}{decision making}
\newacro{AD}{automated driving}
\newacro{PDG}{powered descent guidance}
\newacro{ODE}{ordinary differential equation}
\newacro{RK4}{Runge-Kutta of order 4 integrator}
\newacro{HP}{heat pump}
\newacro{AC}{absorption cooling}
\newacro{ACM}{absorption cooling machine}
\newacro{FC}{free cooling}
\newacro{GBD}{generalized Benders' decomposition}
\newacro{KKT}{Karush-Kuhn-Tucker}

\newacro{NLP}{nonlinear program}
\newacro{MPC}{model predictive control}
\newacro{NMPC}{nonlinear model predictive control}
\newacro{MPPI}{model predictive path integral control}
\newacro{QP}{quadratic program}
\newacro{MIQP}{mixed-integer quadratic program}
\newacro{MILP}{mixed-integer linear program}
\newacro{MINLP}{mixed-integer nonlinear program}
\newacro{MIOCP}{mixed-integer optimal control problem}
\newacro{MI}{mixed-integer}
\newacro{MIP}{mixed-integer program}
\newacro{BB}{branch-and-bound}
\newacro{SQP}{sequential quadratic programming}
\newacro{RNN}{recurrent neural network}
\newacro{OCP}{optimal control problem}
\newacro{LQR}{linear quadratic regulator}
\newacro{iLQR}{iterative linear quadratic regulator}
\newacro{SLQ}{sequential linear quadratic programming}
\newacro{DDP}{differential dynamic programming}
\newacro{MCP}{mixed complementarity problem}
\newacro{IP}{interior point}
\newacro{ADMM}{alternating direction method of multipliers }
\newacro{RTI}{real time iteration}
\newacro{MDP}{Markov decision process}
\newacro{SOC}{second order cone}
\newacro{CIA}{combinatorial integral approximation}
\newacro{STO}{switching time optimization}

\newacro{RL}{reinforcement learning}
\newacro{DP}{dynamic programming}
\newacro{LSTM}{long short-term memory}
\newacro{NN}{neural network}
\newacro{PPO}{proximal policy optimization}
\newacro{PG}{policy gradient}
\newacro{TD}{temporal difference}
\newacro{SAC}{soft actor critic}
\newacro{PI}{policy iteration}
\newacro{SARSA}{state action reward state action}
\newacro{IL}{imitation learning}
\newacro{MLE}{maximum likelihood}
\newacro{SV}{surrounding vehicle}
\newacro{EV}{ego vehicle}

\newacro{FF}{feed forward network}
\newacro{DNN}{deep neural network}

\newacro{UGV}{unmanned ground vehicle}

\newacro{AC4MPC}{Actor Critic for Nonlinear Model Predictive Control}

\newcommand{\R}{\mathbb{R}}

\newcommand{\T}{\top}
\newcommand{\setT}{\mathbb{T}}
\newcommand{\setS}{\mathbb{S}}
\newcommand{\Z}{\mathbb{Z}}
\newcommand{\B}{\mathbb{B}}
\newcommand{\F}{\mathcal{F}}
\newcommand{\Fbar}{\bar{\mathcal{F}}}

\newcommand{\D}{\mathcal{D}}
\newcommand{\prob}{\mathcal{P}}
\newcommand{\NLP}{\mathrm{NLP}}
\newcommand{\MINLP}{\mathrm{MINLP}}
\newcommand{\MILP}{\mathrm{LB-MILP}}

\newcommand{\FNLP}{\mathrm{FNLP}}
\newcommand{\UB}{\mathrm{UB}}
\newcommand{\LB}{\mathrm{LB}}
\newcommand{\norm}[1]{\left\lVert #1 \right\rVert}
\DeclareMathSymbol{\shortminus}{\mathbin}{AMSa}{"39}

\begin{document}

\title{\bf \Large A Sequential Benders-based Mixed-Integer Quadratic Programming Algorithm and Its Implementation in the \texttt{CAMINO} Toolbox}

\author{
        Andrea Ghezzi$^*$ \and Wim Van Roy$^*$ \and \\Sebastian Sager \and Moritz Diehl
}

\titlerunning{A Sequential Benders-based MIQP Algorithm}

\institute{$^*$ The authors contributed equally \\
        A. Ghezzi \at
        Department of Microsystems Engineering (IMTEK) University of Freiburg, Germany \\
        \email{andrea.ghezzi@imtek.uni-freiburg.de}\\
        \emph{Corresponding author}     %
        \and
        W. Van Roy \at
        Department of Mechanical Engineering, KU Leuven, Belgium \\
        \email{wim.vanroy@kuleuven.be}
        \and
        S. Sager \at
        Institute of Mathematical Optimization, Otto-von-Guericke Universit\"at Magdeburg, Germany\\
        \email{sager@ovgu.de}
        \and
        M. Diehl \at
        Department of Microsystems Engineering (IMTEK) and Department of Mathematics, University of Freiburg, Germany \\
        \email{moritz.diehl@imtek.uni-freiburg.de}
}

\date{Received: xxxx / Accepted: xxxx}

\maketitle

\begin{abstract}
        Sequential quadratic programming and sequential convex programming efficiently solve nonlinear programs (NLPs) by linearizing inner nonlinearities while preserving the outer convex structure.
        This paper introduces a sequential mixed-integer quadratic programming (MIQP) algorithm to extend this methodology to mixed-integer nonlinear problems (MINLPs), leveraging the efficiency of modern MIQP solvers.
        The algorithm uses a three-step iterative process.
        First, the MINLP is linearized around the current iterate.
        Second, an MIQP is formulated and solved, with its feasible region restricted to a specific area around the linearization point.
        This region is defined using objective values and derivatives from previous iterations, drawing on concepts from generalized Benders' decomposition.
        Third, the integer variables from the MIQP solution are fixed, and an NLP involving only the continuous variables is solved.
        The best solution among all iterates becomes the linearization point for the next iteration.
        A fallback strategy based on a mixed-integer linear program (MILP) is used when MIQP progress stalls.
        This guarantees convergence to the global optimal solution for convex MINLPs.
        For nonconvex problems, the algorithm functions as a heuristic without global optimality guarantees.
        Numerical experiments show its competitiveness with other MINLP solvers on benchmark problems.
        In addition, the algorithm was successfully applied to mixed-integer optimal control problems, demonstrating its effectiveness in handling challenging nonlinear equality constraints.
        The proposed algorithm is publicly available at \url{https://github.com/minlp-toolbox/CAMINO} with the name \texttt{s-b-miqp}.
\end{abstract}


\section{Introduction}\label{sec: intro}

The class of \acp{MINLP} comprises problems characterized by continuous and discrete variables coupled with nonlinear relationships in the objective or constraints.
Hence, \acp{MINLP} represent a powerful optimization paradigm that offers a natural way to formulate a wide range of problems and applications.
The intersection of nonlinearity and discrete variables poses unique challenges for the numerical solution of such problems, which are characterized by NP-hard complexity \citep{Garey1979,Koppe2011}.
For unbounded support, they are even undecidable, meaning that not every unbounded MINLP problem can be solved.

Several excellent surveys and books cover algorithms for solving  \acp{MINLP} \citep{Belotti2013}, \citep[\S 21]{Sahinidis2019}.
Here, we provide only a high-level picture of the field and focus on the literature directly connected to the algorithm we propose.

One can divide the existing algorithms into two families.
The first consists of branch-and-bound-type algorithms, such as nonlinear branch-and-bound \citep{Dakin1965, Gupta1985} and spatial branch-and-bound \citep{Smith1999}.
The second family relies on the creation of cutting planes to iteratively tighten the integer search space, as in the \ac{GBD} \citep{Geoffrion1972}, outer approximation \citep{Duran1986} and its quadratic version \citep{Fletcher1994}, and extended supporting cutting planes \citep{Westerlund1995, Kronqvist2016}.
A combination of both families includes LP/NLP-based branch-and-bound \citep{Quesada1992,Abhishek2006} and branch-and-check~\citep{Thorsteinsson2001}.
Branch-and-bound-based methods are appealing because they leverage the existence of reliable, efficient, and open-source nonlinear solvers such as IPOPT \citep{Waechter2006} and WORHP \citep{Bueskens2013}.
In contrast, the second family of methods additionally requires efficient mixed-integer linear (quadratic) solvers such as the open-source CBC~\citep{Forrest2005}, SCIP~\citep{Achterberg2009}, Highs~\citep{Huangfu2018}, as well as the commercial CPLEX~\citep{CPLEX}, Gurobi~\citep{Gurobi}, and Mosek~\citep{Mosek}.
Moreover, there also exist solvers that directly implement MINLP-specific algorithms, such as the open-source Bonmin~\citep{Bonami2005}, Couenne~\citep{Belotti2009}, SHOT~\citep{Lundell2022a}, and the commercial Antigone \citep{Misener2014}, Baron \citep{Sahinidis1996}, Knitro~\citep{Byrd2006}, and Gurobi.
A sequential programming algorithms for mixed-integer programs was suggested in~\citep{Exler2007}, resembling a standard trust-region \ac{SQP} method for continuous decision spaces.
The algorithm in~\citep{Exler2007} can be applied to \textit{nonconvex} \acp{MINLP} and also to \acp{MINLP} where the integer variables cannot be relaxed.
However, the authors do not provide a proof of convergence to the global minimizer, even for \textit{convex} \acp{MINLP}.
The Mittelmann benchmarks webpage \citep{MittelmannBenchmark} gives an up-to-date overview of the current solvers and their performance.

Most generic MINLP solvers focus on solving \textit{convex} \acp{MINLP}, often providing only heuristics for \textit{nonconvex} ones.
\andrea{
        In fact, the solution of \emph{nonconvex} MINLPs is much harder, as it requires the construction of global underestimators.
        Most solvers implement a spatial branch-and-bound algorithm, which can be built on top of efficient branch-and-bound codes.
        For decomposition approaches, a method to extend \ac{GBD} for solving \emph{nonconvex} MINLPs is presented in \citep{Li2012}.
        This method decomposes the original \emph{nonconvex} MINLP into convex subproblems by generating tight convex relaxations.
        In general, adopting standard decomposition methods for solving \emph{nonconvex} MINLPs may fail, particularly in the presence of numerous equality constraints.
}
A practical example where these constraints arise is in the solution of \acp{MIOCP} via direct methods, such as direct collocation \citep{Tsang1975} or direct multiple shooting \citep{Bock1984}.
In such cases, equality constraints enforce continuity of the underlying dynamics equations at every grid node.

For the solution of \acp{MIOCP}, an important class of methods known as combinatorial integral approximation (\acs{CIA}\acused{CIA}) \citep{Sager2011a} employs an error-controlled decomposition approach.
In practice, the \ac{MIOCP} is discretized to obtain an \ac{MINLP}, which is then approximately solved by computing an integer approximation.
This approximation minimizes the distance from the relaxed solution of the \ac{MINLP} using a dedicated norm.
An open-source implementation of this approach, named pycombina \citep{Buerger2020a}, has been developed and demonstrated to be effective in various engineering applications \citep{Buerger2021, Robuschi2021}.
This method is capable of handling generic dwell time constraints \citep{Zeile2021} and provides fast approximate solutions, with the main computational challenge lying in the nonlinear optimization rather than combinatorial aspects.
The method relies on the similarity of the relaxed solution with the optimal one.
This is a drawback when dealing with coarse discretization grids or long uptimes, as shown in \citep{Buerger2023, Abbasi2023}.
To address this, \citep{Buerger2023} proposes a new distance function for the second step, based on a quadratic programming approximation around the relaxed \ac{MINLP} solution.
This approach, which is related to a single iteration of the method presented in this paper, often enhances the quality of the \ac{MIOCP} solution in terms of both objective and constraint satisfaction.

Another class of methods for approximately solving \ac{MIOCP} is \ac{STO}.
This approach relies on an initial sequence of states that can be applied to the system, with switching times optimized.
During the iterations, the sequence is adjusted either by inserting additional elements \citep{lee1999control, axelsson2008gradient} or by removing unnecessary sequences \citep{Abbasi2023}.
Also this approach is a heuristic and relies on the initial sequence provided to the solver.

\andrea{
\subsection{Contribution}
In this work, we introduce an algorithm for solving \acp{MINLP}, drawing inspiration from \ac{SQP} and classical \ac{MINLP} methods based on outer approximation.
We obtain a \ac{MIQP} master problem by linearizing the original \ac{MINLP} at the best solution found.
The integer part of the MIQP solution is then fixed in the \ac{MINLP} to obtain a \ac{NLP}.
After each MIQP-NLP iteration, we compute a new \ac{GBD} cut that is used to construct a polytope, named ``Benders region''.
The Benders region further restricts the feasible set of the MIQP to a portion where we expect to obtain an accurate MIQP approximation of the given MINLP.
To guarantee convergence to global minima in case of \emph{convex} MINLPs, we introduce a second master problem that takes the form of a MILP, which contains OA cuts for the best solution found and \ac{GBD} cuts for all the other visited points.
This MILP has the property to underestimate the original \ac{MINLP}, thus providing valid lower bound for the \ac{MINLP} objective.
For \emph{nonconvex} MINLPs, the presented algorithm is an heuristic.}

\andrea{
Differently from quadratic OA~\cite{Fletcher1994}, our MIQP master problem contains a linearization of the constraints of the original \ac{MINLP} only for the best solution found.
The other constraints correspond to the Benders-region, based on \ac{GBD} cuts, and infeasibility (or \emph{no-good}) cuts.
Also, the additional MILP master problem allows us to prove convergence to the global minimizer for convex MINLPs.
We generate infeasibility cuts by projecting the infeasible integer point on the relaxed feasible set rather than minimizing some norm of the constraint violation as done in \cite{Duran1986} or \cite{Fletcher1994}.
In our preliminary work \citep{Ghezzi2023a}, we presented the idea of sequentially solving \acp{MIP} while restricting their feasible set using polytopic regions computed with the information collected during algorithm iterations.
However, in \cite{Ghezzi2023a} the proposed algorithm is a pure heuristic even for \emph{convex} \acp{MINLP}, the polytopic regions do not leverage gradient information as they are based on a Voronoi partitioning, and we did not include a mechanism to handle infeasible \acp{NLP} as we assumed that they were always feasible by using slack variables.
}

A fundamental contribution of this work is the open-source implementation of the proposed algorithm within the package \texttt{CAMINO}\footnote[1]{\url{https://github.com/minlp-toolbox/CAMINO}} (Collection of Algorithms for Mixed-Integer Nonlinear Optimization).
\texttt{CAMINO} is written in Python and relies on CasADi \citep{Andersson2019} to both model optimization problems and interface existing solvers.
The users have the flexibility to choose from various solvers interfaced by CasADi, to solve both \acp{NLP} and \acp{MIP}, enabling an implementation free from reliance on commercial solvers.
In addition to the new algorithm presented in this paper, \texttt{CAMINO} includes a comprehensive range of other existing algorithms, such as GBD, (quadratic) outer approximation, feasibility pumps, and the ability to seamlessly use \ac{MINLP} solvers, already interfaced by CasADi, such as Bonmin.
Further details are available in the \texttt{README} file of the repository.

\subsection{Outline}
In the remainder of this section, we outline some preliminary definitions and notions utilized throughout the rest of the paper.
In Section~\ref{sec: s-b-miqp algorithm}, we introduce the new algorithm, first by giving a short overlook, later we describe in detail all its constituent components.
Moreover, we prove that for \emph{convex} MINLPs the algorithm converges in a finite number of iteration either to the global optimum or to a certificate of infeasibility.
We conclude the section by illustrating the behavior of the proposed algorithm with a simple tutorial example.
In Section~\ref{sec: extension to nonconvex case}, we propose an extension for treating nonconvex \acp{MINLP} by introducing heuristics to modify the generated cutting planes.
We demonstrate that these introduced heuristics do not compromise convergence to the global minimizer in the case of \emph{convex} \acp{MINLP}.
Additionally, for \acp{MINLP} where the integer variables enter affinely, making the relaxed integer feasible set convex, we prove that the proposed algorithm terminates either at feasible solution or with a certificate of infeasibility.
In Section~\ref{sec: results}, we compare the proposed algorithm against \andrea{Bonmin, Gurobi, SCIP, and SHOT} on a large subset of \acp{MINLP} instances from the MINLPLib \citep{Bussieck2003a,MINLPLib} containing at least one integer and one continuous variable.
Moreover, we present the results obtained with the proposed algorithm in two cases of optimal control for switched systems: a textbook example of a small, nonlinear, and unstable system, and a complex nonlinear energy system for building control.
For the latter example, we could only compare against Bonmin and the specialized algorithm \ac{CIA}.
Finally, Section~\ref{sec: conclusion} presents some conclusions and future work directions.

\subsection{Notation}
We denote with $\Z_{[a, b]}$ the set of integer numbers in the interval $[a, b]$ with $a, b \in \Z$ and $a < b$, and with $\Z_{\geq 0}$, the set of nonnegative integer numbers.
For a vector-valued function $f: \R^n \to \R^m$ we denote the transpose of the Jacobian as $\nabla f(z) = \left(\frac{\partial f}{\partial z} (z)\right)^\top$, such that $\nabla f(z) \in \R^{n\times m}$.
The term \textit{convex} \ac{MINLP} refers to a \ac{MINLP} that is convex with respect to both the continuous optimization variables and the relaxed integer variables.

\subsection{Mixed-Integer Nonlinear Problem formulations}\label{sec: formulations}
We consider the generic \ac{MINLP} with $x \in X \subset \R^{n_x}$ and $y \in Y \subset \Z^{n_y}$ of the form
\begin{tcolorbox}[colback=white,colframe=black]
\begin{mini}[2]
        {x \in X, y \in Y}{f(x, y)}
        {\mathcal{P}_\mathrm{MINLP}: \quad \qquad\label{op: generic MINLP}}{}
        \addConstraint{g(x, y)}{\leq 0}
        \addConstraint{h(x, y)}{= 0.}
\end{mini}
\end{tcolorbox}
\noindent Minimizing over the continuous variable $x \in X$ for a fixed $y \in Y$ yields the following parametric NLP
\begin{tcolorbox}[colback=white,colframe=black]
        \begin{mini}[2]
                {x \in X}{f(x, y)}{\prob_\NLP: \quad \qquad\label{op: NLP with fixed integers}}{J(y) \coloneqq}
                \addConstraint{g(x, y)}{\leq 0}
                \addConstraint{h(x, y)}{= 0},
        \end{mini}
\end{tcolorbox}
\noindent Now, we can conceptually state $\prob_\MINLP$  as
\begin{mini}[2]
        {y \in Y}{J(y).}{}{}
\end{mini}
We highlight the definition of $J$ since the proposed algorithm solves $\prob_\MINLP$ exploiting function evaluations and first-order information of $J$ itself.
By slight abuse of notation, we use $J(y)$ to denote either the optimization problem $\prob_\NLP$, its objective function, or a specific objective value.
The intended meaning should be clear from the context.

The proofs in this work rely on the following assumptions regarding problem \eqref{op: generic MINLP}.
\begin{assumption}\label{ass: cont and diff fn, integer finiteness}
        We assume that
         \begin{enumerate}
                \item $Y = \bar{Y} \cap \Z^{n_y}$, where both $X \subset \R^{n_x}$ and $\bar{Y}\subset \R^{n_y}$ are closed convex polyhedral sets.
         	\item Functions $f: \R^{n_x} \times \R^{n_y} \to \R$, $g: \R^{n_x} \times \R^{n_y} \to \R^{n_g}$ and $h: \R^{n_x} \times \R^{n_y} \to \R^{n_h}$ are at least twice continuously differentiable.
                \item The integer set $Y$ is finite.
         \end{enumerate}
\end{assumption}
\begin{assumption}\label{ass: constraint qualification}
        The Mangasarian-Fromovitz constraint qualification holds at the solution $x^\star$ of $\prob_\NLP$ with $J(y)$ for all $y \in Y$.
\end{assumption}

\begin{def_private}
        We define the feasible set of the integer variables as
        \[
                \F \coloneqq \{y \in Y \; | \; \exists \, x \in X, g(x, y) \leq 0, h(x, y) = 0 \}.
        \]
        Similarly, we define the feasible set of the relaxed integer variables as
        \[
                \Fbar \coloneqq \{y \in \bar{Y} \; | \; \exists \, x \in X,  g(x, y) \leq 0, h(x, y) = 0 \}.
        \]
\end{def_private}

\begin{def_private}
        Let $f$ be a continuously differentiable function.
        We define its linearization at $(\bar{x}, \bar{y})$ as a first-order Taylor expansion $f_{\mathrm{L}}(x, y; \bar{x}, \bar{y}) := f(\bar{x}, \bar{y}) + \nabla f (\bar{x}, \bar{y})^\top
        \left(\begin{matrix}
                x - \bar{x} \\
                y - \bar{y}
        \end{matrix}\right)$.
\end{def_private}

\smallskip
\begin{assumption}[Convexity]\label{ass: convexity}
        Function $h$ is affine, and functions $g_1, \dots, g_{n_g}$ and $f$ are convex on $X \times \bar{Y}$.
\end{assumption}
\noindent With Assumption \ref{ass: convexity}, the set $\Fbar$ would be convex and any subgradient of $J$ would provide a lower bound of $J$.

\section{Sequential Benders MIQP algorithm}\label{sec: s-b-miqp algorithm}
The algorithm proposed in this work operates by leveraging successive approximations of the original MINLP problem~\eqref{op: generic MINLP}, which are constructed using first- and second-order derivative information of the problem functions.
At each iteration, a quadratic master problem, denoted as $\mathcal{P}_\mathrm{BR-MIQP}$, is formulated by linearizing $\prob_\MINLP$ around the best solution visited so far, $(x_b, y_b)$.
This approximation is expected to provide high-quality candidate solutions within a defined region $\B$ around the current linearization point.
In parallel, a second master problem, $\mathcal{P}_\MILP$, is constructed using accumulated first-order information from all previously visited points, following a similar philosophy to generalized Benders decomposition~\citep{Geoffrion1972}.
These two master problems jointly address the combinatorial aspects of $\mathcal{P}_\mathrm{MINLP}$ while providing complementary perspectives for exploring the feasible space.
To evaluate the quality of the candidate solutions with respect to the original problem, an auxiliary nonlinear programming problem $\prob_\NLP$ is solved.
We include a second auxiliary problem $\mathcal{P}_\mathrm{FNLP}$, which is solved whenever $\prob_\NLP$ is infeasible.
The overall procedure is illustrated in Figure~\ref{fig: convex sbmiqp flow chart}, where a single iteration is represented by a full cycle from the red diamond labeled ``$\mathrm{LB} < \mathrm{UB}$'' back to itself.
The iteration index is denoted by $k \in \mathbb{Z}_{\geq 0}$.
As evident from the flowchart, the algorithm follows the general structure of outer approximation methods, but distinguishes itself by employing two separate master problems.

\begin{figure}
        \centering
        \includegraphics[width=\textwidth,trim={0 0 0.5cm 0},clip]{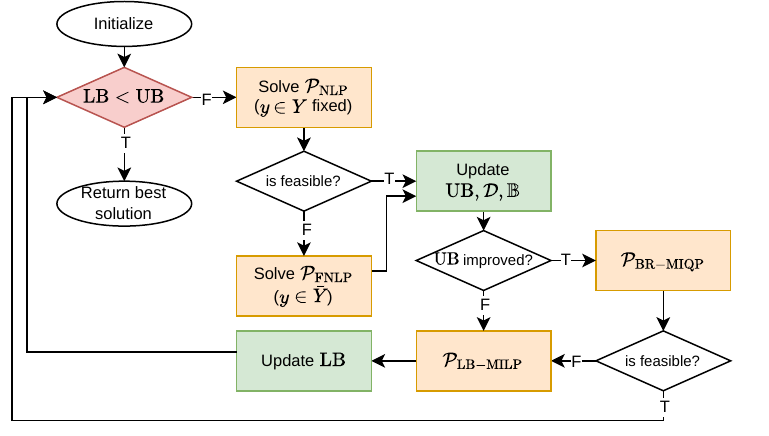}
        \caption{\andrea{High-level flow chart of the proposed algorithm.}
        The orange rectangles highlight the optimization problems solved during the algorithm iterations, the green rectangles perform operations on the data obtained by solving the different problems, and the red diamond highlights the termination condition.
        The outcome ``T'' of the diamonds corresponds to ``True'', and ``F'' corresponds to ``False''.
        For a rigorous description, see the pseudo-code of Algorithm~\ref{alg: benders region s-miqp}. \label{fig: convex sbmiqp flow chart}}
\end{figure}

Moving forward, we first present the constituent components of the algorithm highlighted in Figure~\ref{fig: convex sbmiqp flow chart}, namely the auxiliary problems, the master problems, and the termination condition.
Secondly, by means of pseudo-code, we present the algorithm in detail and derive its theoretical properties.
We close the section with a tutorial example illustrating the algorithm's behavior.

\subsection{Computation of $\nabla J(y)$ from the solution of the auxiliary problem $\prob_\NLP$}\label{subsec: auxiliary problem}
The auxiliary problem $\prob_\NLP$ \eqref{op: NLP with fixed integers}, introduced in Sec.~\ref{sec: formulations}, is obtained by fixing the integer variables of the original \ac{MINLP} \eqref{op: generic MINLP}.
For a given $\tilde{y} \in Y$, solving $\prob_\NLP$ yields the optimal value $J(\tilde{y})$.
If $\prob_\NLP$ is infeasible, we set $J(\tilde{y}) = +\infty$.
\andrea{
        The function $J$ is generally nonsmooth and nonconvex due to changes in the active set and possible nonuniqueness of optimal solutions of $\prob_\NLP$.
        Consequently, classical differentiability cannot be expected, and generalized derivatives must be employed.
}

\andrea{Under Assumption~\ref{ass: constraint qualification} (MFCQ), the set of Lagrange multipliers associated with a local solution is nonempty and bounded, and the value function $J$ is locally Lipschitz continuous around $\tilde y$.
To characterize the first-order sensitivity of $J$ with respect to $y$, we consider a lifted formulation of $\prob_\NLP$ in which $y$ is treated as an optimization variable and fixed by an equality constraint}
\begin{mini}[2]
    {x \in X, y \in Y}{f(x, y)}{}{\label{op: NLP with lifted y}}
    \addConstraint{g(x, y)}{\leq 0}
    \addConstraint{h(x, y)}{= 0}
    \addConstraint{y - \tilde{y}}{= 0.}
\end{mini}
\andrea{Let $(x^\star,\tilde y)$ be a locally optimal solution of
\eqref{op: NLP with lifted y}, and let
$\lambda \in \R^{n_g}$,
$\mu \in \R^{n_h}$, and
$\mu_{\tilde y} \in \R^{n_y}$
denote the Lagrange multipliers associated with the inequality constraints,
equality constraints, and the constraint $y-\tilde y=0$, respectively.
The corresponding \ac{KKT} conditions are}
\begin{align}
    \begin{cases}
        \nabla_x f(x^\star, \tilde{y}) + \nabla_x g(x^\star, \tilde{y}) \lambda + \nabla_x h(x^\star, \tilde{y}) \mu &= 0, \\
        \nabla_y f(x^\star, \tilde{y}) + \nabla_y g(x^\star, \tilde{y}) \lambda + \nabla_y h(x^\star, \tilde{y}) \mu &= -\mu_{\tilde{y}}, \\
        g(x^\star, \tilde{y}) &\leq 0, \\
        h(x^\star, \tilde{y}) &= 0, \\
        \lambda &\geq 0, \\
        \lambda_i g_i(x^\star, \tilde{y}) &= 0, \quad i=1, \dots, n_g.
    \end{cases}
\end{align}
\andrea{
\begin{proposition}[Clarke subgradient of the value function under MFCQ]
\label{prop:clarke-subgradient}
Suppose that MFCQ holds at $(x^\star,\tilde y)$ for problem \eqref{op: NLP with lifted y}, and that $(x^\star,\tilde y)$ is a locally optimal solution.
Then the value function $J$ is locally Lipschitz continuous around $\tilde y$ and
\[
    -\mu_{\tilde y} \in \partial_C J(\tilde y),
\]
where $\partial_C J(\tilde y)$ denotes the Clarke subdifferential of $J$ at
$\tilde y$.
\end{proposition}
}
\andrea{
\begin{proof}
Under MFCQ, the set of Lagrange multipliers of \eqref{op: NLP with lifted y} is nonempty and bounded.
Moreover, classical sensitivity results for parametric nonlinear programs imply that the value function is locally Lipschitz and that the multiplier associated with the parameter-fixing constraint characterizes the first-order variation of the optimal value.
See, e.g., \cite[Ch. 3]{Fiacco1983} for a detailed proof.
\end{proof}
}
\andrea{
As a consequence of Proposition~\ref{prop:clarke-subgradient}, we define}
\begin{equation}
    \nabla J(\tilde y) \coloneqq -\mu_{\tilde y},
\end{equation}
\andrea{
where $\nabla J(\tilde y)$ is interpreted as an element of the Clarke subdifferential $\partial_C J(\tilde y)$.}
\begin{remark}
\andrea{Because $J$ is nonsmooth, $\nabla J(\tilde y)$ is not a gradient in the classical sense.
When $\prob_\NLP$ is solved using an interior-point method, the computed multipliers approximate KKT multipliers of \eqref{op: NLP with lifted y}, and hence $\nabla J(\tilde y)$ represents an approximation of a Clarke subgradient.}
In practice, this approximation has not caused significant issues in our numerical experiments.
\end{remark}

\subsection{The feasibility auxiliary problem -- $\prob_\FNLP$}\label{subsec: restoration phase FNLP}

For some $\hat{y} \in Y$, the problem $\prob_\NLP$ which we aim to solve might be infeasible.
In this specific situation, we solve a feasibility problem to find the closest point to $\hat{y}$ that lies on the boundary of $\Fbar$.
The sought point is obtained by solving the following NLP
\begin{tcolorbox}[colback=white,colframe=black]
\parbox{0.2\columnwidth}{
                \raggedleft
                $\mathcal{P}_\FNLP(\hat{y}, y_b):$}
\parbox{0.7\columnwidth}{
\begin{mini!}[2]
        {x \in X, y \in \bar{Y}}{\norm{y - \hat{y}}_2^2}{\label{op: FNLP}}{}
        \addConstraint{g(x, y)}{\leq 0}
        \addConstraint{h(x, y)}{= 0}
        \addConstraint{\norm{y_b - y}_2^2}{\leq \norm{y_b - \hat{y}}_2^2,}{\quad \text{if } y_b \text{ given}.\label{cns: FNLP ball around y_b}}
\end{mini!}
}
\end{tcolorbox}
Constraint \eqref{cns: FNLP ball around y_b} is enforced only if a feasible (best) solution $y_b \in \F$ is available.
This constraint narrows the feasible set $\bar{Y}$ by requiring that $\bar{y}$ lies within a ball of radius $\norm{y_b - \hat{y}}_2$ around the best point $y_b$.
This requirement is particularly helpful in case of a disconnected feasible set $\bar{\mathcal{F}}$, but it does not introduce further complexity in case of a convex feasible set $\Fbar$ (cf. Lemma \ref{lemma: infeasibility}).

\andrea{
\begin{remark}
By construction, $Y = \bar Y \cap \mathbb Z^{n_y}$, and hence $Y \subseteq \bar Y$.
Consequently, the feasible set of the integer problem satisfies $\F \subseteq \Fbar$, independently of convexity or regularity assumptions.
In particular, if $\Fbar \neq \emptyset$, then $\mathcal P_\mathrm{FNLP}(\hat y,y_b)$ is feasible, since it reduces to the projection of $\hat y$ onto the nonempty set $\Fbar$.
If $\Fbar = \emptyset$, then $\mathcal P_\mathrm{FNLP}$ and hence the original MINLP are infeasible.
\end{remark}
}

We denote by $\bar{y}(\hat{y}, y_b)$ the integer component of the solution of $\mathcal{P}_\FNLP$.
Moreover, we utilize $\bar{y}(\hat{y}, y_b)$ to construct the following infeasibility cut
\begin{equation}\label{eq: infeasibility cut}
        (\hat{y} - \bar{y})^\T (y - \bar{y}) \leq 0.
\end{equation}
This constraint can be interpreted as the hyperplane on a convex set going through $\bar{y}$ with a normal vector $\hat{y} - \bar{y}$.
An illustration of the infeasibility cut is given in Figure~\ref{fig: infeasibility cut for convex MINLP}.
\begin{figure}
        \centering
        \includegraphics[]{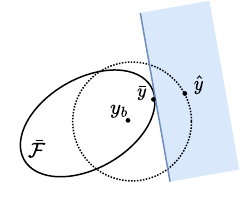}
        \caption{In lightblue an example of infeasibility cut for a convex $\prob_\MINLP$, i.e., the relaxed feasible set $\bar{\F}$ is a convex set and corresponds to the interior of the solid black line. $y_b \in \F$ is the current best and feasible solution, $\hat{y} \in \Z^{n_y}$ is an infeasible solution, and $\bar{y}$ is the solution of $\prob_\FNLP$. \label{fig: infeasibility cut for convex MINLP}}
    \end{figure}

\begin{lemma}\label{lemma: infeasibility}
        \andrea{Let $\hat y \in Y$ be such that $\hat y \notin \Fbar$, and let $\bar y$ be a solution of $\mathcal P_\mathrm{FNLP}(\hat y,y_b)$.
        Then $y=\hat y$ violates the infeasibility cut $(\hat y - \bar y)^\top (y - \bar y) \le 0$.}
\end{lemma}
\begin{proof}
        \andrea{
        By definition of $\mathcal{P}_\FNLP(\hat{y}, y_b)$, the point $\bar{y}$ satisfies $\|\bar{y} - y_b\|_2^2 \le \|\hat{y} - y_b\|_2^2$ whenever $y_b$ is available.
        Moreover, since $\bar{y}$ minimizes $\|y - \hat{y}\|_2^2$ over the feasible set, we have $\bar{y} \neq \hat{y}$ whenever $\hat{y} \notin \Fbar$.}
        Thus, for $y= \hat{y}$, $(\hat{y} - \bar{y})^\top(\hat{y} - \bar{y}) = \|\hat{y} - \bar{y}\|_2^2 > 0$, which implies that $\hat{y}$ violates constraint~\eqref{eq: infeasibility cut}.
\end{proof}

\begin{lemma}\label{lemma: infeasibility convex set}
        \andrea{Under Assumption~\ref{ass: convexity}, let $\hat y \notin \Fbar$, and let $\bar y \in \arg\min_{y \in \Fbar} \|y - \hat y\|_2^2$.
        Then, the inequality $(\hat{y} - \bar{y})^\T (y - \bar{y}) \leq 0$ is valid for all $y \in \Fbar$.}
\end{lemma}

\begin{proof}
        \andrea{Since $\bar y$ is the Euclidean projection of $\hat y$ onto the closed convex set $\Fbar$,} the first-order optimality condition of the projection problem states that, at the optimal point $\bar{y}$, the gradient $(y - \hat{y})$ evaluated at $\bar{y}$ must satisfy
        \[
        (\hat{y} - \bar{y})^\T (y - \bar{y}) \leq 0 \quad \forall y \in \Fbar,
        \]
        which is exactly the infeasibility cut~\eqref{eq: infeasibility cut}.
\end{proof}
Based on Lemma~\ref{lemma: infeasibility convex set}, we can conclude that the constraint~\eqref{eq: infeasibility cut} is satisfied for the best found solution $y_b$, if the problem is convex.
On the contrary, for a nonconvex $\Fbar$, we remark that constraints \eqref{cns: FNLP ball around y_b} may be active.
Thus, the corresponding infeasibility cut may not be tangential to $\Fbar$.

\subsection{Bookkeeping, lower bound function and Benders region}
Before introducing the two master problems, we present which data needs to be stored during algorithm iterations.
Then, we show how these data are used for constructing $J_\LB$, a function that outer approximates all the visited points and is relevant for the master problem $\prob_\MILP$.
Lastly, using $J_\LB$ we define the ``Benders region'' $\B$, which restricts the feasible set of $\prob_\mathrm{BR-MIQP}$ to a neighborhood of the current best solution, where we expect $\prob_\mathrm{BR-MIQP}$ to approximate $\prob_\MINLP$ well.

Consider a collection of points for which first-order information of $J$ is available in terms of evaluation tuples $(i, x_i, y_i, J(y_i)$, $\nabla J(y_i))$.
We collect them in a dataset
\[
        \D_k \coloneqq \left( (0, x_0, y_0, J(y_0), \nabla J(y_0)), \dots, (k, x_k, y_k, J(y_k), \nabla J(y_k))\right).
\]

At each iteration $k$, we classify the outcome and store the results accordingly.
If we successfully solve $\prob_\NLP$ and obtain a feasible solution with objective value $J(y_k)$, we add $k$ to the index set $\setT_k$, which tracks all feasible iterations.
Otherwise, if we solve the feasibility problem $\mathcal{P}_\mathrm{FNLP}(y_k, y_{b(k)})$ and obtain solution $\bar{y}_k$, we store this infeasible solution and add $k$ to the index set $\setS_k$. Note that $\setT_k \cup \setS_k = \Z_{[0, k]}$.

To track the best solution found up to iteration $k$, we define:
\begin{equation}\label{eq: expression of b(k)}
        b(k) \in
\begin{cases}
    \text{arg } \underset{i \in \setT_k}{\min} \; J(y_i), & \text{if } \setT_k \neq \emptyset, \\
    \text{arg } \underset{i \in \setS_k}{\min} \; \Vert \bar{y}_i - \hat{y}_i \Vert_2^2, & \text{otherwise}.
\end{cases}
\end{equation}
When multiple indices achieve the same minimum value, we select the smallest index to ensure uniqueness. Clearly, $b(k) \leq k$.
If $b(k) \in \setT_k$, then $(x_{b(k)}, y_{b(k)})$ represents the \textit{incumbent} solution, i.e., the best feasible solution found so far, with objective value $\UB = J(y_{b(k)})$.

The lower bound function $J_\mathrm{LB}$ is defined on the index sets $\setT_k, \setS_k$ and on the collected data $\D_k$ as follows
\begin{mini}[2]
        {\eta}{\eta}{\label{op: J_LB}}{{J_{\mathrm{LB}}(y; \setT_k, \setS_k, \D_k) \coloneqq}}
        \addConstraint{\eta}{\geq J(y_i) + \nabla J(y_i)^\top(y - y_i),}{\quad i \in \setT_k}
        \addConstraint{(y_j - \bar{y}_j)^\top (y - \bar{y}_j)}{\leq 0,}{\quad j \in \setS_k.}
\end{mini}
The term $\nabla J(y_i)$ is a subgradient of the potentially nonsmooth function $J$, and its computation has been discussed in Section \ref{subsec: auxiliary problem}.

\andrea{Consistent with the above discussion on restricting the search to a neighborhood of the incumbent solution,
we define the Benders region at iteration $k$ as}
\begin{align}\label{eq: benders region definition}
        \B_k \coloneqq &\{y \in \R^{n_y} \; | \; J_\mathrm{LB}(y;\setT_k, \setS_k, \D_k) \leq \bar{J}_k\},
\end{align}
where $\bar{J}_k$ is determined by
\begin{equation}\label{eq: reduced rhs J_bar}
        \bar{J}_k \coloneqq \alpha J(y_{b(k)}) + (1-\alpha) \mathrm{LB}, \quad \alpha \in [0, 1),
\end{equation}
where $\alpha$ is a user-chosen parameter, and $\LB$ is the value of the best lower bound of MINLP~\eqref{op: generic MINLP} found by the algorithm.
The computation of this reduced right-hand-side term follows an idea from~\citep{Kronqvist2020}.
It has the property to exclude all the visited points including the best point $y_{b(k)}$ from the current region $\B_k$.
Therefore, a point $y \in \B_k$ must be outer approximated by the linear cuts constructed using previous solutions and must have an objective lower than $\bar{J}_k$.

\subsection{Master problems: Benders-region mixed-integer quadratic program  -- $\prob_\mathrm{BR-MIQP}$}\label{subsec: benders tr computation}
The first master problem we introduce is a \ac{MIQP} since it can provide a more accurate approximation of $\prob_\MINLP$ compared to a \ac{MILP}, ultimately requiring fewer outer approximation iterations to achieve a solution.
Moreover, \ac{MIQP} master problems have been shown to be favorable over \ac{MILP} master problems in the context of outer approximation.
In fact, it is easy to construct simple examples where outer approximation with \ac{MILP} master problems requires complete enumeration to find the global optimum~\citep{Fletcher1994}.
A disadvantage of \ac{MIQP} master problems argued in~\citep{Fletcher1994} regards longer computation times compared to \ac{MILP}.
However, modern \ac{MIQP} solvers have significantly improved in efficiency, narrowing the runtime gap between \ac{MILP} and \ac{MIQP}.

In particular, the ``Benders-region MIQP'', $\mathcal{P}_\mathrm{BR-MIQP}$, incorporates a quadratic approximation of $\prob_\MINLP$ computed around the current best point $(x_{b(k)}, y_{b(k)})$ with the Hessian approximation $B_{b(k)}$.
\andrea{The matrix $B_{b(k)}$ is a symmetric and positive semidefinite approximation of the Hessian of the Lagrangian of $\prob_\NLP$ with respect to $(x,y)$, evaluated at $(x_{b(k)},y_{b(k)})$.
In practice, $B_{b(k)}$ may correspond to the exact Hessian of the Lagrangian of $\prob_\MINLP$ evaluated using primal and dual information of the best solution found, or to quasi-Newton updates, or set to zero, yielding a first-order model.}
Problem $\mathcal{P}_\mathrm{BR-MIQP}$ also includes additional constraints imposed by the Benders region \eqref{eq: benders region definition} and infeasibility cuts \eqref{eq: infeasibility cut}.
Hence, denoting the linearizations of functions $f, g, h$ at $(\bar{x}, \bar{y})$ as $f_\mathrm{L}(\cdot \;; \bar{x}, \bar{y}), g_\mathrm{L}(\cdot \;; \bar{x}, \bar{y})$, and $h_\mathrm{L}(\cdot \;; \bar{x}, \bar{y})$, respectively, we state
\begin{tcolorbox}[colback=white,colframe=black]
        ${\mathcal{P}_\mathrm{BR-MIQP}: }$
	\begin{mini}[2]
		{x \in  X , y \in  Y}{f_\mathrm{L}(x, y; {x}_{b(k)}, {y}_{b(k)}) +
			\frac{1}{2}  \left(  \begin{matrix} x  -  {x}_{b(k)} \\ y  -  {y}_{b(k)} \end{matrix}  \right)^{ \top}    B_{b(k)} \left( \begin{matrix} x  -  {x}_{b(k)} \\ y  -  {y}_{b(k)} \end{matrix} \right)}{\label{op: main-BR-MIQP-expanded}}{}
		\addConstraint{g_\mathrm{L}(x, y; {x}_{b(k)}, {y}_{b(k)})}{\leq 0}
		\addConstraint{h_\mathrm{L}(x, y; {x}_{b(k)}, {y}_{b(k)})}{= 0}
		\addConstraint{J(y_i) + \nabla J(y_i)^\top   (y-y_i)}{\leq \bar{J}_k,}{\quad i \in \setT_k}
                \addConstraint{(y_i - \bar{y}_i)^\top (y - \bar{y}_i)}{\leq 0,}{\quad i \in \setS_k.}
	\end{mini}
\end{tcolorbox}
We denote the objective value of \eqref{op: main-BR-MIQP-expanded} as $V_\mathrm{MIQP}$.
In the case no feasible solution has been found yet, i.e., $\setT_k = \emptyset$, the functions $f_\mathrm{L}, g_\mathrm{L}, h_\mathrm{L}$ are obtained by linearizing at the infeasible point $(x_{b(k)}, y_{b(k)})$, where $b(k)$ is computed according to \eqref{eq: expression of b(k)}.
Problem $\mathcal{P}_\mathrm{BR-MIQP}$ depends on the convex polyhedron $\B_k$.
There are no guarantees that an integer point $y$ such that $y \in Y \cap \B_k$ exists, thus $\mathcal{P}_\mathrm{BR-MIQP}$ might be infeasible.
When this happens, we trigger the if-condition of line \ref{alg: fallback strategy} and solve a different master problem, named $\prob_\MILP$, which is presented in the next section.

\subsection{Master problems: Lower-bound \ac{MILP} -- $\prob_\MILP$}\label{subsec: master problem lb-milp}

We present the second master problem named ``Lower-bound MILP'', $\prob_\MILP$, which has a structure similar to the master problem of the \ac{GBD} algorithm \citep{Geoffrion1972}.
In our case, we additionally include in the constraints the linear approximation of the original MINLP \eqref{op: generic MINLP} around the best solution found.
This can be seen as a mix between \ac{GBD} and linear outer approximation \citep{Fletcher1994}.
Problem $\prob_\MILP$ is introduced because a quadratic outer approximation method cannot guarantee convergence to a global optimum even for convex MINLP as already shown in \cite{Fletcher1994}.
In fact, the solution of $\mathcal{P}_\mathrm{BR-MIQP}$ (or an alternative \ac{MIQP} master problem) does not provide an outer approximation of $\mathcal{P}_\mathrm{MINLP}$.
Specifically, solving $\prob_\MILP$ lets us establish whether the current best solution is optimal or other solutions exist.

At the $k$-th iteration of the proposed algorithm, $\mathcal{P}_\mathrm{LB-MILP}$ can be formulated as a MILP in the full variable space as follows
\begin{tcolorbox}[colback=white,colframe=black]
$\mathcal{P}_\mathrm{LB-MILP}:$
\begin{mini!}[2]
                {\substack{\eta \in \R, \; x \in X, \; y \in Y}}{\eta}{\label{op: benders master problem}}{}
                \addConstraint{\eta }{\geq f_\mathrm{L}(x, y; x_{b(k)}, y_{b(k)}),}{\andrea{\quad \text{if } b(k) \in \setT_k}\label{cns: LB-MILP f_L OA}}
                \addConstraint{0 }{\geq g_\mathrm{L}(x, y; x_{b(k)}, y_{b(k)}),}{\quad \text{if } b(k) \in \setT_k \label{cns: LB-MILP g_L OA}}
                \addConstraint{0 }{= h_\mathrm{L}(x, y; x_{b(k)}, y_{b(k)}),}{\quad \text{if } b(k) \in \setT_k \label{cns: LB-MILP h_L OA}}
                \addConstraint{\eta }{\geq J(y_i) + \nabla J(y_i)^\top (y-y_i),}{\quad i \in \setT_k \andrea{\setminus \{b(k)\}} \label{cns: benders constraint in benders master problem}}
                \addConstraint{0}{\geq (y_i - \bar{y}_i)^\top (y - \bar{y}_i),}{\quad i \in \setS_k.\label{cns: LB-MILP inf cuts}}
\end{mini!}
\end{tcolorbox}
We denote the objective value of \eqref{op: benders master problem} as $V_\mathrm{MILP}$.
Unlike $\prob_\mathrm{BR-MIQP}$, the outer approximation \andrea{objective cut \eqref{cns: LB-MILP f_L OA}} and constraint cuts \eqref{cns: LB-MILP g_L OA}-\eqref{cns: LB-MILP h_L OA} are only imposed if the current best solution $(x_{b(k)}, y_{b(k)})$ is feasible, i.e., $b(k) \in \setT_k$.
\andrea{
        This ensures that no cuts are imposed when the variables $y$ are fractional.
}
In $\prob_\MILP$, for each feasible visited solution $y_i$ we construct a linear approximator in the integer space of the nonlinear function $J: Y \mapsto \R$, cf., \eqref{op: NLP with fixed integers}.
Thus, by minimizing the slack variable $\eta$, we aim to find the minimum of the epigraph of the function obtained by the intersection of all the epigraphs of linear models.
\andrea{
        Since we impose the OA constraints \eqref{cns: LB-MILP f_L OA}-\eqref{cns: LB-MILP h_L OA}, the Benders constraints \eqref{cns: benders constraint in benders master problem} are imposed for all the indices in $\setT_k$ but $b(k)$ if $b(k) \in \setT_k$.
        Note that the set minus operation is defined such that if $b(k) \notin \setT_k$ then $\setT_k \setminus \{b(k)\} = \setT_k$.
}
Under Assumption~\ref{ass: convexity}, the objective value $V_\mathrm{MILP}$ of $\prob_\MILP$ at an optimal solution provides a lower bound on the objective value of $\prob_\MINLP$.
Therefore, the lower bound $\LB$ is updated whenever the solution of $\prob_\MILP$ yields a $V_\mathrm{MILP}$ higher than the current $\LB$.
Eventually, in case $\prob_\MILP$ is infeasible we set $V_\mathrm{MILP} = +\infty$.
Infeasibility of $\prob_\MILP$ might happen if at the iteration $k$ the algorithm has tested every value $y \in Y$ and $\setT_k = \emptyset$, for further details see Section~\ref{subsec: theory sbmiqp}.

\subsection{The S-B-MIQP algorithm}\label{sec: algorithm}
Algorithm~\ref{alg: benders region s-miqp} presents the complete Sequential Benders MIQP (S-B-MIQP) algorithm in detail.
First, in this subsection, we consider Assumption~\ref{ass: convexity} about convexity to hold; thus, line~\ref{alg:modification} is omitted and we assume $\tilde{\cal D}_k \coloneqq  {\cal D}_k$.
We address the nonconvex case in Section \ref{sec: extension to nonconvex case}.

Consistent with the notation adopted earlier, in Algorithm~\ref{alg: benders region s-miqp}, the letter $k$ denotes the iteration index, where $k \in \Z$.
Algorithm~\ref{alg: benders region s-miqp} begins with an integer point $y_0 \in Y$ provided by the user, along with a lower bound, $\LB$, obtained by solving the integer relaxation of \eqref{op: generic MINLP}, while the upper bound is $\UB = +\infty$.
Following initialization, we enter a cycle that terminates only when the lower bound becomes equal to the upper bound, $\UB$.
This termination condition is common for mixed-integer programming algorithms.

The first half of the cycle involves evaluating the quality of the integer solution $y_k$ and solving for the continuous variable $x$.
In each iteration, the upper bound is updated by comparing the current $\UB$ with $J(y_k)$, the objective value of $\prob_\NLP$.
In case $\prob_\NLP$ is infeasible for a given $y_k$, we solve the feasibility problem $\prob_\FNLP$, whose solution is used to construct specific cutting planes aiming to steer the integer solution back to the feasible set.
The information obtained in each iteration is stored in $\D_k$. The index sets $\setT_k$ and $\setS_k$ are updated based on the feasibility of $\prob_\NLP$.

The second half of Algorithm~\ref{alg: benders region s-miqp} focuses on computing new integer solutions and lower bounds.
The integer solutions are computed in the master problems $\prob_\mathrm{BR-MIQP}$ and $\prob_\mathrm{LB-MILP}$.
The idea is to approach the minimizer of $\prob_\MINLP$ \eqref{op: generic MINLP} by solving \acp{MIQP}, as done in \ac{SQP} for continuous problems.
\andrea{
        Specifically, $\prob_\mathrm{BR-MIQP}$ is solved if a new incumbent solution is found in the current or in the previous iteration, cf. line \ref{alg: solving bmiqp} of Algorithm~\ref{alg: benders region s-miqp}.
        In case, the incumbent solution is not improving or $\prob_\mathrm{BR-MIQP}$ becomes infeasible due to tighter Benders regions, Algorithm~\ref{alg: benders region s-miqp} switches to solving $\prob_\mathrm{LB-MILP}$.
        The solution of $\prob_\mathrm{LB-MILP}$ provides valid lower bound on the objective of $\prob_\MINLP$ \eqref{op: generic MINLP}.
        The solution of $\prob_\mathrm{BR-MIQP}$ is resumed only if the integer solution found by $\prob_\mathrm{LB-MILP}$ produces a new incumbent solution.
        In this case, Algorithm~\ref{alg: benders region s-miqp} attempts to solve a new $\prob_\mathrm{BR-MIQP}$ constructed around the new incumbent solution.
}

\begin{algorithm}
	\caption{Sequential Benders MIQP (S-B-MIQP)}~\label{alg: benders region s-miqp}
	\begin{algorithmic}[1]\normalsize
		\State \textbf{Initialize}: $y_0 \in Y$, $\B_0 \gets Y$, $k \gets 0$, $\setS_{-1} \gets \emptyset$, $\setT_{-1} \gets \emptyset$, $\UB= + \infty$, $\LB= - \infty$, $b(k) \gets 0$, \andrea{\texttt{needMILP} $\gets$ \texttt{False}}
                \State $\LB \gets \min_{y \in \bar{Y}} J(y)$ \Comment {\footnotesize{\ttfamily LB given by the objective of \ac{MINLP} relaxation}} \normalsize
		\While {$\LB < \UB$}:
                \State Given $y_k$ solve $\prob_\NLP$
                \If {$\prob_\NLP$ is feasible}
                        \State Store solution $(k, x_k, y_k, J(y_k), \nabla J(y_k))$ in $\D_k$ and $\setT_k \gets \setT_{k-1} \cup \{k\}$
                        \If {$J(y_k) < \UB$}
                                \State $\UB \gets J(y_k)$ \Comment {\footnotesize{\ttfamily Update upper bound}} \normalsize
                        \EndIf
                \Else
                        \State Solve $\mathcal{P}_\FNLP(y_k, y_{b(k)})$ and get $\bar{y}_k$
                        \State Store solution $(k, x_k, y_k, \bar{y}_k, \Vert \bar{y}_k - y_k \Vert_2^2)$ in $\D_k$ and $\setS_k \gets \setS_{k-1} \cup \{k\}$
                \EndIf
                \State Compute $b(k)$ according to \eqref{eq: expression of b(k)}
                \State Modify gradients in $\D_k$, get $\tilde{\D}_k$  \label{alg:modification} \Comment {\footnotesize{\ttfamily Required only for noncvx MINLP (cf. Sec.~\ref{sec: extension to nonconvex case})}} \normalsize
                \State Compute Benders region $\B_k$ based on $\tilde{\D}_k$ according to \eqref{eq: benders region definition}

                \If {\andrea{$k - b(k) \le 1$}\label{alg: solving bmiqp}} \Comment {\footnotesize{\ttfamily \andrea{Best solution found in the last iteration}}} \normalsize
                        \State Solve $\mathcal{P}_\mathrm{BR-MIQP}$ given $(x_{b(k)}, y_{b(k)})$, $B_{b(k)}$, and $\B_k$. Get its solution $\tilde{y}$.
                \Else
                        \State \andrea{\texttt{needMILP} $\gets$ \texttt{True}}
                \EndIf
                \If {\andrea{\texttt{needMILP} is \texttt{True}} \textbf{or} $\mathcal{P}_\mathrm{BR-MIQP}$ is infeasible \label{alg: fallback strategy}}:
                \State Solve $\mathcal{P}_\mathrm{LB-MILP}$ with $J_\mathrm{LB}(y;\setT_k, \setS_k, \tilde{\D}_k)$, get solution $\tilde{y}$ and $V_\mathrm{MILP}$
                \If {$\mathcal{P}_\mathrm{LB-MILP}$ is infeasible}
                \State $\LB \gets +\infty$
                \Else
                \If {$V_\mathrm{MILP} > \LB$}
                \State $\LB \gets V_\mathrm{MILP}$ \Comment {\footnotesize{\ttfamily Update lower bound}} \normalsize
                \EndIf
                \EndIf
                \State \andrea{\texttt{needMILP} $\gets$ \texttt{False}}
                \EndIf
                \State $y_{k+1} \gets \tilde{y}$
		\State $b(k+1) \gets b(k)$
                \State $k \gets k+1$
                \EndWhile{}
	\end{algorithmic}
\end{algorithm}

\subsection{Algorithm properties}\label{subsec: theory sbmiqp}
\andrea{
Before stating the theoretical properties of Algorithm~\ref{alg: benders region s-miqp}, we emphasize that the proofs rely on standard results from convex analysis and outer-approximation theory.
In particular, Assumption~\ref{ass: convexity} guarantees that first-order Taylor expansions yield global underestimators, while Assumption~\ref{ass: constraint qualification} ensures existence and boundedness of Lagrange multipliers and validity of first-order optimality conditions.
}

\begin{lemma}\label{lemma: lower bounds}
        A continuously differentiable function $f: X \to \R$ is convex if and only if $X$ is a convex set and $f(z) \geq f(x) + \nabla f(x)^\top(z-x)$ holds for all $x, z \in X$.
\end{lemma}
\begin{proof}
        Given in \citep[\S 3.1.3]{Boyd2004}.
\end{proof}

\begin{lemma}\label{lemma: equiv LB-MILP and MINLP}
        \andrea{
        Suppose Assumptions~\ref{ass: cont and diff fn, integer finiteness} and~\ref{ass: convexity} hold.
        Assume further that all integer points have been visited at iteration $k$, i.e., $Y = \{ y_i \mid i \in \setT_k \cup \setS_k \}$.
        Then the lower-bound master problem $\mathcal P_\mathrm{LB-MILP}$ has the same set of optimal integer solutions as $\mathcal P_\mathrm{MINLP}$.
        }
\end{lemma}
\begin{proof}
        Proved in \citep[\S 13.1]{Li2006} for \ac{GBD}.
        \andrea{
                Here, the only difference is that the master problem $\mathcal{P}_\mathrm{LB-MILP}$ imposes OA constraints about $(x_{b(k)}, y_{b(k)})$ if $b(k) \in \setT_k$, while in the standard GBD master problem we would have a Benders cut like \eqref{cns: benders constraint in benders master problem} also for $(x_{b(k)}, y_{b(k)})$.
                Thus, the feasible set of $\mathcal{P}_\mathrm{LB-MILP}$ tighter compared to the one of the standard master problem in \ac{GBD}.
                Below, we provide the logical steps of the proof.
        }

\andrea{
        Under Assumption~\ref{ass: convexity}, the functions $f$ and $g_1,\dots,g_{n_g}$ are convex and $h$ is affine on $X \times \bar Y$.
        Hence, by Lemma~\ref{lemma: lower bounds}, their first-order Taylor expansions define global underestimators on $X \times \bar Y$.
        Therefore, every feasible solution of $\mathcal P_\mathrm{MINLP}$ is feasible for $\mathcal P_\mathrm{LB-MILP}$.
        This shows that $\mathcal P_\mathrm{LB-MILP}$ is an outer approximation of $\mathcal P_\mathrm{MINLP}$.
        Moreover, for each feasible integer point $y_i \in \setT_k$, the constraint $\eta \ge J(y_i) + \nabla J(y_i)^\top (y - y_i)$ is included in $\mathcal P_\mathrm{LB-MILP}$.
        Since all feasible integer points have been visited, these inequalities describe the epigraph of $J$ over $\F$ exactly.
        For each infeasible integer point $y_i \in \setS_k$, the infeasibility cut
        $(y_i - \bar y_i)^\top (y - \bar y_i) \le 0$ excludes $y_i$ from the feasible region of $\mathcal P_\mathrm{LB-MILP}$.
        Hence, the feasible integer assignments of $\mathcal P_\mathrm{LB-MILP}$ coincide exactly with $\F$.
        Consequently, minimizing $\eta$ over $\mathcal P_\mathrm{LB-MILP}$ is equivalent to minimizing $J(y)$ over $Y$.
        Therefore, the two problems share the same set of optimal integer solutions.
}
\end{proof}

\begin{theorem}\label{th: algorithm convergence}
        If Assumptions \ref{ass: cont and diff fn, integer finiteness}, \ref{ass: constraint qualification} and \ref{ass: convexity} hold, Algorithm \ref{alg: benders region s-miqp} \andrea{terminates in a finite number of iterations and returns either a global optimal solution of $\mathcal P_\mathrm{MINLP}$ or a certificate of infeasibility.}
\end{theorem}

\begin{proof}
        \andrea{
        We first show that no integer assignment is repeated.
        If $\prob_\NLP$ is infeasible at $y_k$, then the infeasibility cut $(y_k-\bar y_k)^\top(y-\bar y_k)\le 0$ excludes $y_k$ from subsequent master problems.
        If $\prob_\NLP$ is feasible and $\mathcal P_\mathrm{BR-MIQP}$ returns a new integer solution, then by construction $\tilde y \ne y_k$.
        If $\mathcal P_\mathrm{BR-MIQP}$ is infeasible or stagnates, the algorithm solves $\mathcal P_\mathrm{LB-MILP}$.
        If $\mathcal P_\mathrm{LB-MILP}$ returns $\tilde y = y_k$, then the stopping condition $\LB \ge \UB$ is satisfied and the algorithm terminates.
        Hence, no integer assignment is revisited.
        Since $Y$ is finite by Assumption~\ref{ass: cont and diff fn, integer finiteness}, finite termination follows.}

        \andrea{
        We now prove correctness under convexity.
        Under Assumption~\ref{ass: convexity} and Lemma~\ref{lemma: lower bounds}, the linearizations of $f$, $g$, and $h$ are global underestimators.
        Therefore, $\mathcal P_\mathrm{LB-MILP}$ is an outer approximation of $\mathcal P_\mathrm{MINLP}$.
        Consequently, every feasible solution of $\mathcal P_\mathrm{MINLP}$ remains feasible for $\mathcal P_\mathrm{LB-MILP}$.
        }
        Let $(x^\star, y^\star)$ be the global optimal solution of $\mathcal{P}_\mathrm{MINLP}$ with objective $f^\star$, hence $\prob_\NLP$ is feasible for $J(y^\star)$ and $\UB = J(y^\star) = f(x^\star, y^\star)$.
        \andrea{
                Since $(x^\star,y^\star)$ is feasible for $\mathcal P_\mathrm{MINLP}$, it is also feasible for $\mathcal P_\mathrm{LB-MILP}$ by the outer approximation property.
        Evaluating the epigraph constraints of $\mathcal P_\mathrm{LB-MILP}$ at $(x^\star,y^\star)$ yields}
        \begin{subequations}
        \begin{align}
                        \eta &\geq f(x^\star, y^\star) + \nabla f(x^\star, y^\star)^\top \left(\begin{matrix}
                                \tilde{x} - x^\star \\ y^\star - y^\star
                        \end{matrix}\right), \label{eq: proof lower bounds oa objective cut} \\
                        \eta &\geq J(y_i) + \nabla J(y_i)^\top (y^\star - y_i), \quad i \in \setT_k \andrea{\setminus \{b(k)\}}. \label{eq: proof lower bounds benders cut}
                \end{align}
        \end{subequations}
        From \eqref{eq: proof lower bounds oa objective cut}, we know that no valid descent direction exists along $x$ at $x^\star$ (cf. Assumption \ref{ass: constraint qualification}), therefore
        \[
                \nabla f(x^\star, y^\star)^\top \left(\begin{matrix}
                        \tilde{x} - x^\star \\ 0
                \end{matrix}\right) \geq 0.
        \]
        \andrea{By convexity of $J$, each Benders cut \eqref{eq: proof lower bounds benders cut} satisfies
        \[
                J(y_i) + \nabla J(y_i)^\top (y^\star - y_i) \le J(y^\star) = f(x^\star, y^\star) = \UB, \quad \text{for all } i \in \setT_k \setminus \{b(k)\},
        \]
        }
        it follows
        \[
                \LB = \eta \geq f(x^\star, y^\star) + \nabla f(x^\star, y^\star)^\top \left(\begin{matrix}
                        \tilde{x} - x^\star \\ 0
                \end{matrix}\right) \geq \UB,
        \]
        thus Algorithm~\ref{alg: benders region s-miqp} terminates.

        Assume Algorithm~\ref{alg: benders region s-miqp} terminates at a point $(x',y')$ which is not the global optimum, hence with objective value $J(y')$ \andrea{ strictly greater than $f(x^\star,y^\star)$.
        Termination implies $\LB \ge \UB$.
        Since $\UB = J(y')$, we have $\LB \ge J(y') > f(x^\star,y^\star)$.
        However, $(x^\star,y^\star)$ is feasible for $\mathcal P_\mathrm{LB-MILP}$, so $\LB \le f(x^\star,y^\star)$.
        This is a contradiction.
        }

        Algorithm~\ref{alg: benders region s-miqp} classifies $\mathcal{P}_\mathrm{MINLP}$ as infeasible when every $y \in Y$ has proved to make $\prob_\NLP$ infeasible.
        \andrea{
        By definition of $\F$, this implies $\F = \emptyset$.
        Hence, $\prob_\MINLP$ is infeasible.
        Moreover, for each $y \in Y$ an infeasibility cut excludes $y$ from the feasible set of $\mathcal{P}_\mathrm{LB-MILP}$.
        Since $Y$ is finite, after all integer points have been processed, no feasible integer assignment remains.
        Therefore $\mathcal P_\mathrm{LB-MILP}$ becomes infeasible and the Algorithm~\ref{alg: benders region s-miqp} terminates with $\LB = \UB = +\infty$.
        }
\end{proof}

\subsection{Tutorial example for S-B-MIQP}
We illustrate the behavior of Algorithm \ref{alg: benders region s-miqp} for a MINLP where we can have an effective graphical representation. The example is taken from \citep{Ghezzi2023a}.
Consider the following convex~MINLP
\begin{mini!}
    {\substack{x \in \mathbb{R}, \\(y_{[1]}, y_{[2]}) \in \mathbb{Z}^2}}{(y_{[1]} - 4.1)^2 + (y_{[2]} - 4.0)^2 + \lambda x}
    {\label{op: tutorial MINLP}}{\label{op: cost function of tutorial MINLP}}
	\addConstraint{y_{[1]}^2 + y_{[2]}^2 - r^2 - x}{\leq 0}{}\label{op: slacked cns}
	\addConstraint{-x}{\leq 0},
\end{mini!}
where $r = 3$ and $\lambda = 1000$.
The term $\lambda x$ in \eqref{op: cost function of tutorial MINLP} can be seen as a penalization of the violation of the quadratic constraint $y_{[1]}^2 + y_{[2]}^2 - r^2 \leq 0$.
The shape of the cost function and the global minimizer of \eqref{op: tutorial MINLP} are represented in Figure \ref{fig: contour nlp}.
The global minimizer can be found graphically and corresponds to $(x, y_{[1]}, y_{[2]}, x) = (0, 2, 2)$.
Note that $x$ is determined by $x = \max(0, y_{[1]}^2 + y_{[2]}^2 - r^2)$.
\begin{figure}
    \centering
    \includegraphics[]{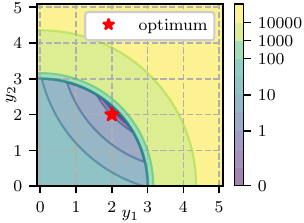}
    \caption{Level lines of the cost function of problem \eqref{op: tutorial MINLP}. The red asterisk denotes the global minimizer.}~\label{fig: contour nlp}
\end{figure}
Given a linearization point $(\bar{x}, \bar{y}_{[1]}, \bar{y}_{[2]})$, one can compute the corresponding master problems, $\prob_{\mathrm{BR-MIQP}}$ and $\prob_{\mathrm{LB-MILP}}$.
Since the original cost is convex and quadratic, one can choose $B$ as the exact Hessian of \eqref{op: cost function of tutorial MINLP}.
The quadratic constraint \eqref{op: slacked cns} is linearized to fit the MILP/MIQP approximation as follows
\begin{align} \label{cns: linearized cns}
	- \bar{y}_{[1]}^2 - \bar{y}_{[2]}^2 + 2 y_{[1]} \bar{y}_{[1]} + 2 y_{[2]} \bar{y}_{[2]} - x \leq r^2.
\end{align}
Also, for the MILP/MIQP approximation $x$ is implicitly defined as $x = \max(0, -\bar{y}_{[1]}^2 - \bar{y}_{[2]}^2 + 2 y_{[1]} \bar{y}_{[1]} + 2 y_{[2]} \bar{y}_{[2]} - r^2)$.
For this reason, in the following, we focus only on the values of $y_{[1]}, y_{[2]}$.
We stack $y_{[1]}, y_{[2]}$ in the vector $y_k$ as $y_k \coloneqq (y_{[1]}, y_{[2]})$ where the subscript $k$ denotes the iteration number of Algorithm~\ref{alg: benders region s-miqp}.

We apply Algorithm~\ref{alg: benders region s-miqp} to solve \eqref{op: tutorial MINLP}, with hyper-parameter $\alpha = 0.9$ for the reduced right-hand-side \eqref{eq: reduced rhs J_bar}.
Table \ref{tab: tutorial example iteration} reports the results of each iteration, and Figure \ref{fig: illustrative example} gives a graphical interpretation.
\begin{table}
    \centering
    \ra{1.2}
    \caption{Iterations of Algorithm \ref{alg: benders region s-miqp} for problem \eqref{op: tutorial MINLP}}
    \begin{tabular}{@{}lcccccc@{}} \toprule
        $k$     & $\LB$ & $\UB$         & $b(k)$        & $y_k$         & $J(y_k)$      & $V_k$ \\ \midrule
        0       & 7.44  & 7016.81       & 0             & (0, 4)        & 7016.81       & -    \\
        1       & 7.44  & 7016.81       & 0             & (4, 3)        & 16001.01      & 1.01    \\
        2       & 7.44  & 4005.21       & 2             & (3, 2)        & 4005.21       & 5.21    \\
        3       & 8.41  & 8.41          & 3             & (2, 2)        & 8.41       & 8.41    \\

        \bottomrule
    \end{tabular}\label{tab: tutorial example iteration}
\end{table}

\begin{figure}
    \centering
    \includegraphics[width=0.95\columnwidth]{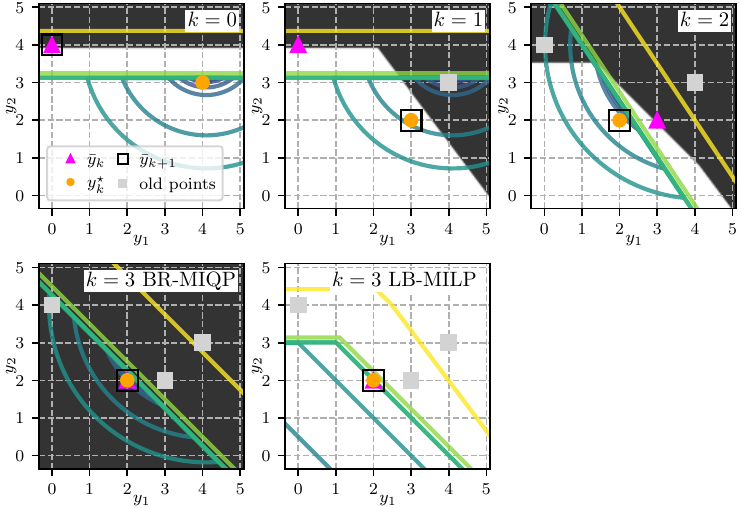}
    \caption{
        Representation of algorithm iterations.
	The level lines correspond to either $\prob_\mathrm{BR-MIQP}$ or $\prob_\mathrm{LB-MILP}$, according to the iteration. The darkened areas describe the regions excluded by the Benders regions $\mathbb{B}_k$.
	}~\label{fig: illustrative example}
    \vspace*{-0.3cm}
\end{figure}

The problem is initialized with $y_0 = (0, 4)$, and the lower bound computed via MINLP relaxation is $\LB = 7.44$.
In the first iteration, the initial point $y_0$ is evaluated by solving $J(y_0)$.
Since the problem is feasible, its objective value updates the initial upper bound, $\UB = 7016.81$.
Then, we compute the Benders region $\B_0$ and we solve $\prob_\mathrm{BR-MIQP}$ with linearization point $(x_0, y_0)$ and Hessian $B_0$.
Since $\prob_\mathrm{BR-MIQP}$ is feasible, we store its objective in $V_1$ and its solution in $y_{[1]}$, hence $V_1 = 1.01$ and $y_{[1]} = (4, 3)$.
The next two iterations of Algorithm \ref{alg: benders region s-miqp}, i.e., $k=1, 2$ have a similar structure to the first iteration.
We move our focus to the fourth iteration where $y_3 = (2,2)$, and the associated $J(y_3)$ objective is $8.41$, then $\UB$ is updated.
Also, the Benders region $\B_3$ is computed, and its intersection with the integer feasible set is empty.
Therefore, we switch to the solution of the associated $\prob_\MILP$, whose minimizer is $y_{[2]} = (2, 2)$, and the objective is $V_\mathrm{MILP} = 8.41$.
The lower bound $\LB$ is updated and, finally, the termination condition of Algorithm~\ref{alg: benders region s-miqp} is met since $\UB = \LB$.

\section{Extension to the nonconvex case}\label{sec: extension to nonconvex case}

Under the convexity Assumption \ref{ass: convexity}, Theorem \ref{th: algorithm convergence} guarantees termination of Algorithm \ref{alg: benders region s-miqp} either to the global optimal solution of the convex MINLP or with a certificate of infeasibility.
However, when we deal with nonconvex \acp{MINLP}~\eqref{op: generic MINLP}, for which we only require Assumptions~\ref{ass: cont and diff fn, integer finiteness} and \ref{ass: constraint qualification} to hold, the termination condition of Algorithm~\ref{alg: benders region s-miqp} might be triggered in \andrea{unintended} situations.
Due to nonconvexities, the cutting planes based on $J_\LB$ are not guaranteed to be lower bounds for the global optimum or for the current best solution.

We delineate two premature termination scenarios.
The first scenario occurs when the solution of $\prob_\MILP$ yields a new point $(x_k, y_k)$ with a value exceeding the current $\UB$.
This situation is inherently ambiguous because we halt the algorithm due to a newfound minimizer that we acknowledge to be inferior to the best point found during the algorithm's iterations.
The second scenario arises from infeasibility cutting planes, causing $\prob_\MILP$ to become infeasible.
Here, we set $\LB$ to $+\infty$, triggering Algorithm~\ref{alg: benders region s-miqp} termination.
This second scenario is particularly misleading as it encroaches upon the condition $\LB=+\infty$, typically reserved for detecting infeasibility in the original problem $\prob_\MINLP$.

We illustrate the first situation of premature termination with the following example.
Consider the nonlinear integer program
\begin{mini}
        {y \in \Z_{[-4, 4]}}{\left( y^2 - 5 \right)^2 + 4 y.}{}{\label{example: benders counterexample}}
\end{mini}
For simplicity, we run Algorithm~\ref{alg: benders region s-miqp} with zero Hessian $B=0$ in $\mathcal{P}_\mathrm{BR-MIQP}$, resulting in an \ac{MILP}.
As shown in Figure \ref{fig: benders deadlock for noncvx minlp}, at the last iteration, the new Benders region $\B_3$ is empty. Hence, we attempt to solve $\prob_\MILP$.
Note that for the latter problem, the cutting planes are not lower bounds of the current best solution $y_{b(3)}=-3$.
The solution of $\prob_\MILP$ is $y=2$, which has a value greater than $\UB$, triggering the termination of Algorithm~\ref{alg: benders region s-miqp}.

The second scenario is illustrated in Figure \ref{fig: infeasibility cut nonconvex MINLP}, where $\Fbar$ is nonconvex, $y_b \in \Fbar$ is the best solution found, and $\hat{y}$ is the infeasible solution that Algorithm~\ref{alg: benders region s-miqp} has computed lastly.
In the attempt to create an infeasibility cut as \eqref{eq: infeasibility cut} for $\prob_{\FNLP}(\hat{y}, y_b)$, we cut a portion of the feasible set that includes $y_b$.
This makes $\LB = \infty$, triggering the termination of Algorithm~\ref{alg: benders region s-miqp}.

\begin{figure}
        \centering
        \includegraphics[width=\columnwidth]{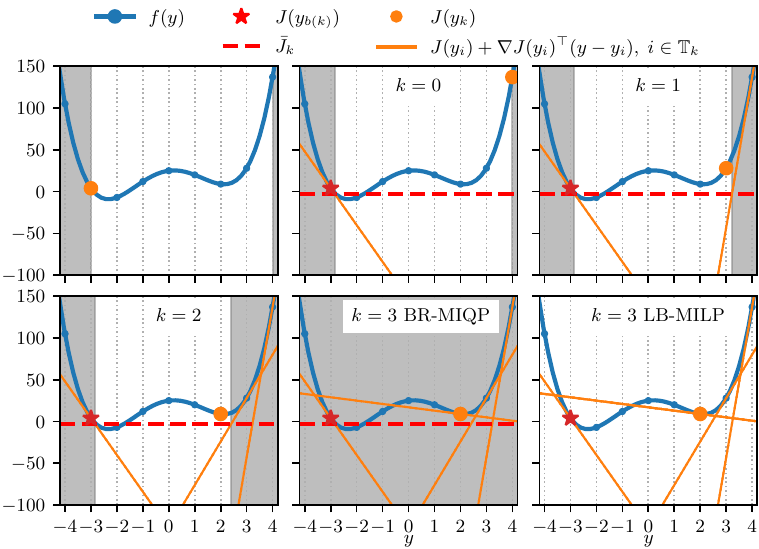}
        \caption{Iterations of Algorithm~\ref{alg: benders region s-miqp} for Problem~\eqref{example: benders counterexample} with $\LB$ equal to the global minimum and $\alpha=0.5$, starting from $y_0=-3$. The shaded dark areas correspond to the area excluded from the feasible set by the given Benders region $\B_k$. The first plot depicts the initial condition of Algorithm~\ref{alg: benders region s-miqp}, and the last plot illustrates its termination, highlighting the first scenario of premature termination where the solution of $\prob_\MILP$ has value higher than the current $\UB$.}\label{fig: benders deadlock for noncvx minlp}
\end{figure}

\begin{figure}
        \centering
        \includegraphics[width=0.5\columnwidth]{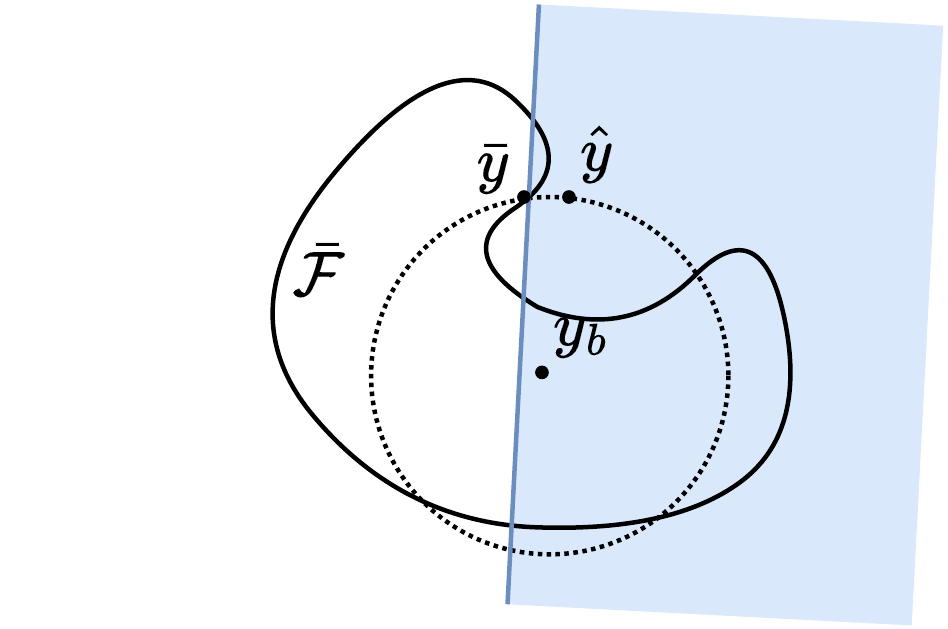}
        \caption{Termination condition triggered by infeasibility cuts. The relaxed integer feasible set $\Fbar$ corresponds to the interior of the solid black line, $y_b \in \F$ is the current best and feasible solution, $\hat{y} \in \Z^{n_y}$ is infeasible, and $\bar{y}$ is the solution of \eqref{op: FNLP}. The infeasibility cut in $\bar{y}$ makes $y_b$ infeasible and triggers the termination condition.}\label{fig: infeasibility cut nonconvex MINLP}
\end{figure}

\andrea{In what follows, we introduce heuristic safeguarding strategies that aim to prevent premature termination in the nonconvex case, without restoring the
global validity guarantees enjoyed under convexity.}
First, we present a procedure to correct the gradient of the linear models, then we show a way to enlarge the Benders region.

\subsection{Gradient correction}\label{subsec: gradient correction strategy 1}
\andrea{
During the computation of the underestimator $J_{\LB}$ at the $k$-th iteration of Algorithm~\ref{alg: benders region s-miqp}, linearizations of the value function are constructed using vectors $\nabla J(y_i) \in \partial_C J(y_i)$ obtained from auxiliary problems at points $y_i, i \in \setT_k$.
}
\andrea{
In contrast to the convex case, Clarke subgradients of the generally nonconvex value function $J$ provide only local first-order sensitivity information and do not, in general, define globally valid Benders cuts.
As a consequence, the linearization induced by $\nabla J(y_i)$ may overestimate the value function at other points, including the current best incumbent solution $y_{b(k)}$.
}
Specifically, it may occur that for some $i \in \setT_k$,
\begin{equation}\label{eq: gradient correction issue}
J(y_i) + \nabla  J(y_i)^\top (y_{b(k)} - y_i) > J(y_{b(k)}),
\end{equation}
in which case the corresponding linearization is not a valid Benders cut and may prematurely terminate Algorithm~\ref{alg: benders region s-miqp}.

\andrea{To ensure that the underestimator remains valid, we replace such local linearizations by corrected cutting planes that underestimate the value function at the incumbent solution.
}
Specifically, we require the cutting plane associated with $y_i$ to satisfy
\begin{equation}
\label{eq: BTR correct inequality}
J(y_i) + \tilde g^\top (y_{b(k)} - y_i) \le J(y_{b(k)}),
\end{equation}
where $\tilde g$ denotes a cut-generating vector.

The set of admissible cut-generating vectors is therefore defined as
\begin{equation}
\mathbb{G}_{(i,k)} \coloneqq \left\{ \tilde g \;\middle|\; J(y_i) + \tilde g^\top (y_{b(k)} - y_i) \le J(y_{b(k)}) \right\}.
\end{equation}
\andrea{Elements of $\mathbb{G}_{(i,k)}$ are not required to belong to the Clarke subdifferential of $J$.
}

\andrea{Among all admissible cut-generating vectors, we select the one that minimally deviates from the local sensitivity vector $\nabla J(y_i)$} in a weighted Euclidean norm with symmetric positive definite weight matrix $W$.
This leads to the optimization problem
\begin{argmini}
{g \in \mathbb{G}_{(i,k)}}
{\frac{1}{2}\Vert g - \nabla J(y_i) \Vert_W^2}
{\label{op: minimal gradient correction}}
{g^\mathrm{corr}_{(i,k)} \in}
\end{argmini}
By default, the identity matrix is used as weight matrix $W$.

\begin{lemma}
\label{lemma: grad corr computation}
Problem~\eqref{op: minimal gradient correction} admits a closed-form solution.
Let $\Delta y_{(i,k)} = y_{b(k)} - y_i$.
Let $r_{(i,k)} = J(y_{b(k)}) - J(y_i) - \nabla J(y_i)^\top \Delta y_{(i,k)}$.
Then the corrected cut-generating vector is given by
\begin{align}
g^\mathrm{corr}_{(i,k)} =
\nabla J(y_i)
+
\begin{cases}
0,
& \text{if } r_{(i,k)} \ge 0, \\[0.6em]
\displaystyle
\frac{r_{(i,k)}}{\Delta y_{(i,k)}^\top W^{-1} \Delta y_{(i,k)}}
W^{-1} \Delta y_{(i,k)},
& \text{if } r_{(i,k)} < 0.
\end{cases}
\end{align}
\end{lemma}

\begin{proof}
        Let us write problem \eqref{op: minimal gradient correction} in the following form
        \begin{mini}
                {\Delta g}{\frac{1}{2} \Vert{\Delta g}\Vert_W^2}{}{}
                \addConstraint{\Delta g^\top \Delta y - r }{\leq 0,}
        \end{mini}
        where $\Delta g = g^\mathrm{corr}_{(i,k)} -  \nabla J (y_i)$, $\Delta y = y_{b(k)} - y_i$, and $r = J(y_{b(k)}) - J(y_i) - \nabla J (y_i)^\top \Delta y$.
        The Lagrangian function of the problem is given by
        \begin{equation}
                \mathcal{L}(\Delta g, \lambda) = \frac{1}{2} \Vert{\Delta g}\Vert_W^2 + (\Delta g^\top \Delta y - r) \lambda,
        \end{equation}
        and the \ac{KKT} system is given by
        \begin{align}
                \begin{cases}
                        W \Delta g + \lambda \Delta y & = 0, \\
                        \Delta g^\top \Delta y - r &\leq 0, \\
                        \lambda &\geq 0, \\
                        \lambda (\Delta g^\top \Delta y - r) &= 0.
                \end{cases}
        \end{align}
        In order to solve the system, we distinguish two cases:
        \begin{enumerate}
                \item if the constraint is inactive or weakly active $\Delta g^\top \Delta y - r \leq 0$ and $\lambda = 0$. The optimal solution can be directly obtained from the stationarity condition and $\Delta g = 0$.
                \item if the constraint is strictly active $\Delta g^\top \Delta y - r = 0$ and $\lambda > 0$. From the stationarity condition we obtain $\Delta g = -\lambda W^{-1} \Delta y$. Substituting $\Delta g$ into the primal feasibility condition we get $\lambda = - \frac{r}{\Delta y^\top W^{-1}\Delta y}$. By substitution, we find $\Delta g = \frac{r}{\Delta y^\top W^{-1}\Delta y} W^{-1} \Delta y$.
        \end{enumerate}
\end{proof}
When the best point does not change in the current $k$-th iteration, we only need to check if the new point $y_k$ verifies inequality \eqref{eq: BTR correct inequality} with $\tilde{g} \equiv \nabla J(y_k)$ at the incumbent solution.
If necessary, we correct its gradient.
However, when $y_k$ is the new best point, we must verify if inequality \eqref{eq: BTR correct inequality} holds for all points $y_i, i \in \setT_k$ stored in $\D_k$ and correct the problematic gradients.

\andrea{\begin{remark}
The correction step deliberately sacrifices the subgradient interpretation in order to guarantee global validity of the Benders underestimator at the incumbent solution.
In the convex case, where subgradients yield globally valid cuts, the correction is inactive.
In the nonconvex setting, the procedure can be interpreted as a projection of a local sensitivity vector onto the set of admissible Benders cuts for the incumbent solution.
\end{remark}
}

We denote by ${\cal D}^\mathrm{corr}_k$ the ``corrected'' dataset at iterate $k$, which contains the corrected cut-generating vectors $g^\mathrm{corr}_{(i,k)}$ instead of the local sensitivity vectors $\nabla J(y_i)$.

\begin{lemma}\label{lemma: gradient correction for convex MINLP}
	Under Assumptions~\ref{ass: cont and diff fn, integer finiteness} and~\ref{ass: convexity}, the corrected gradients equal the original ones, i.e., ${\cal D}^\mathrm{corr}_k = {\cal D}_k$.
\end{lemma}
\begin{proof}
        This directly follows from Lemma \ref{lemma: lower bounds}.
        Since $J$ is a convex function with convex domain $Y$, \andrea{the Clarke subdifferential coincides with the convex subdifferential.}
        Hence, it holds $J(y_i) + \nabla J(y_i)^\top(y_j - y_i) \leq J(y_j)$, for any $y_i, y_j \in Y$.
\end{proof}

\subsection{Region expansion via gradient amplification}\label{subsec: gradient amplification}
When the gradient correction is computed in a nonconvex situation, we find the minimum correction that ensures that the best point found is outer approximated by every cut, which potentially makes such point the only one feasible for $\prob_\MILP$.
The gradient correction resolves the ambiguous termination described at the beginning of this section, but it can dramatically limit Algorithm~\ref{alg: benders region s-miqp} from further exploring the integer solution space.
For this purpose, we introduce a constant value $\rho \geq 1$, which amplifies all the gradients of the available linear model as follows:
\begin{equation}\label{eq: grad expansion}
        g^\mathrm{ampl}_{(i,k, \rho)}  \coloneqq \rho g^\mathrm{corr}_{(i,k)}.
\end{equation}
The amplification factor $\rho$ is a hyper-parameter of Algorithm~\ref{alg: benders region s-miqp}, which could, for example, be chosen offline and kept fixed at runtime. More elaborate strategies to choose $\rho$ separately per inequality and iteration index are also possible but are beyond our interest in this work.

We denote by ${\cal D}^\mathrm{ampl}_k$ the dataset where the corrected and amplified gradients replace the original gradients.
This set is used as the modified dataset in line \ref{alg:modification} of Algorithm \ref{alg: benders region s-miqp}, i.e., the final version of Algorithm~\ref{alg: benders region s-miqp} sets $\tilde{\cal D}_k \coloneqq  {\cal D}^\mathrm{ampl}_k$.

\begin{figure}
        \centering
        \includegraphics[width=\textwidth]{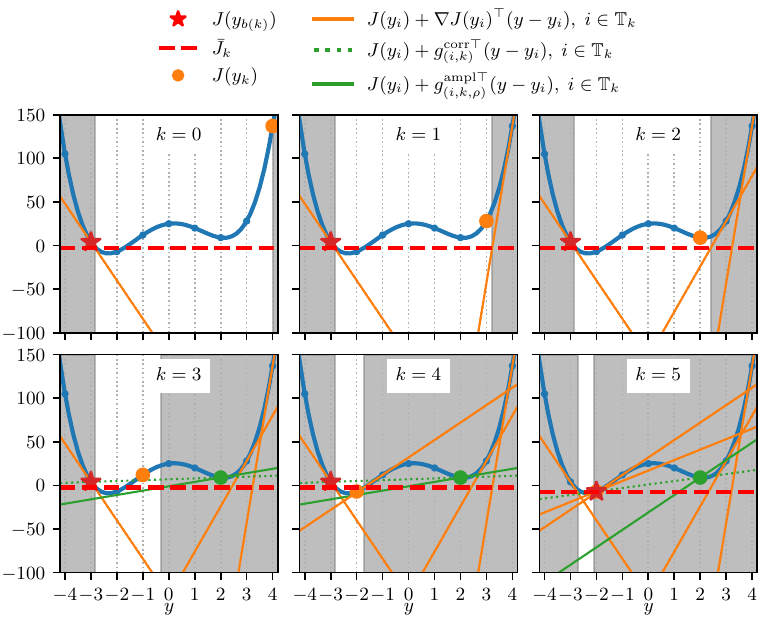}
        \caption{We consider the same problem depicted in Figure \ref{fig: benders deadlock for noncvx minlp} and same settings for Algorithm~\ref{alg: benders region s-miqp}, i.e., hyper-parameter $\alpha=0.5$. Here, we overcome the deadlock via the gradient correction and amplification, choosing $\rho = 5$. The gray area corresponds to infeasible values for the Benders region $\B_k$ constraint.
        }\label{fig: tr exclusion w-fn}
\end{figure}
In Figure \ref{fig: tr exclusion w-fn} we illustrate how the combination of gradient correction and gradient amplification can solve a deadlock situation caused by nonconvexity.
\begin{remark}\label{remark: no amplification for convex MINLP}
        We amplify by factor $\rho$ only the gradients that are corrected. From this, it follows that for convex \acp{MINLP}, there is no amplification, i.e., $\rho = 1$, since there is no gradient correction, cf., Lemma \ref{lemma: gradient correction for convex MINLP}.
\end{remark}

At the end of this section, we demonstrate that incorporating gradient correction and amplification does not compromise the properties of Algorithm \ref{alg: benders region s-miqp}, namely, its finiteness and convergence to a global minimum under Assumption \ref{ass: convexity}.

\subsection{Correction of the infeasibility cuts}
As shown in Figure~\ref{fig: infeasibility cut nonconvex MINLP}, the infeasibility cuts may cause early termination of Algorithm~\ref{alg: benders region s-miqp} by excluding the current feasible and best solution.
To prevent this situation, we perform a cut correction similar to the one introduced in Sec.~\ref{subsec: gradient correction strategy 1}.
Again, we modify the \andrea{normal vector} of the infeasibility cut such that at iteration $k$ the following inequality holds
\begin{equation}\label{eq: corrected infeasibility cuts}
\tilde{n}_{(i, k)}^\T (y_{b(k)} - \bar{y}_i) \leq 0, \quad i \in \setS_k,
\end{equation}
where the corrected \andrea{normal vector} $\tilde{n}_{(i, k)}$ corresponds to the minimal correction of the original \andrea{cut normal} in a weighted Euclidean norm, similarly to \eqref{op: minimal gradient correction}.
Hence,
\begin{argmini}
{n \in \mathbb{N}_{(i,k)} }{\frac{1}{2}\Vert{ n - (\hat{y}_i-\bar{y}_i) }\Vert_W^2}{}{n^\mathrm{corr}_{(i,k)} \in},
\end{argmini}
where $W \succ 0$ and $\mathbb{N}_{(i,k)} \coloneqq \{\tilde{n} \; \vert \; \tilde{n}^\T (y_{b(k)} - \bar{y}_i) \leq 0\}$.
We store the corrected \andrea{cut normals} $n^\mathrm{corr}_{(i,k)}$ in the dataset ${\cal D}^\mathrm{corr}_k$.
We emphasize that we correct the infeasibility cuts exclusively when Algorithm~\ref{alg: benders region s-miqp} has found at least one feasible solution $y_{b(k)} \in \F$.
In case no feasible solution is found, we enforce the infeasibility cuts as defined in \eqref{eq: infeasibility cut}.
Figure~\ref{fig: correction infeasibility cut nonconvex MINLP} illustrates the correction of the infeasibility cut for the scenario depicted earlier in Figure~\ref{fig: infeasibility cut nonconvex MINLP}.

\begin{lemma}\label{lemma: correction for infeasibility cuts in convex MINLP}
Under Assumptions~\ref{ass: cont and diff fn, integer finiteness} and~\ref{ass: convexity}, the corrected \andrea{cut normals} of the infeasibility cuts equal the original ones, i.e., ${\cal D}^\mathrm{corr}_k = {\cal D}_k$.
\end{lemma}
\begin{proof}
The proof follows similarly to Lemma~\ref{lemma: infeasibility convex set}.
\end{proof}

\begin{figure}[h!]
\centering
\includegraphics[width=0.5\columnwidth]{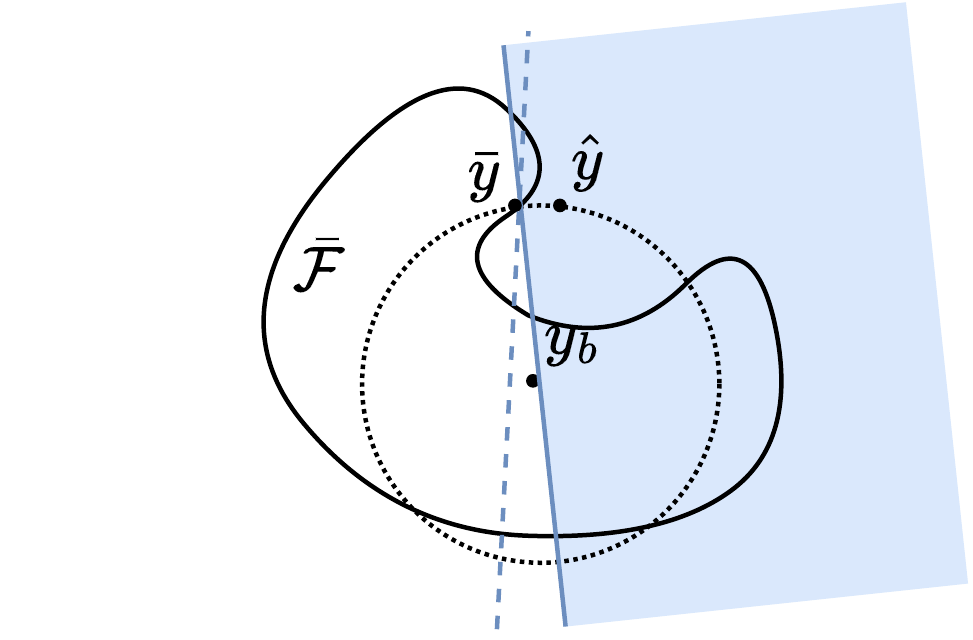}
\caption{Correction of the infeasibility cut of Figure~\ref{fig: infeasibility cut nonconvex MINLP} such that $y_b$ remains feasible.
The corrected cut is depicted solid blue, while the original cut is dashed blue.}
\label{fig: correction infeasibility cut nonconvex MINLP}
\end{figure}

\subsection{Properties of the introduced techniques}

\begin{lemma}\label{lemma: finiteness of sbmiqp with corrections}
If the integer set $Y$ is finite (Assumption~\ref{ass: cont and diff fn, integer finiteness}), Algorithm~\ref{alg: benders region s-miqp} enhanced with gradient correction and amplification procedure for the Benders cuts, and correction of the infeasibility cuts, stops within a finite number of iterations.
\end{lemma}
\begin{proof}
\andrea{Infeasibility cuts~\eqref{eq: corrected infeasibility cuts} preserve the exclusion of previously visited infeasible points whenever a feasible incumbent exists, while ensuring that the current feasible best solution is not excluded}
If a feasible solution is not found, it might happen that an infeasibility cut makes the search space an empty set.
Therefore, $\LB = +\infty$ and Algorithm~\ref{alg: benders region s-miqp} terminates.
If a feasible solution is found, the modified Benders cuts $J(y_i) + \rho \tilde{g}^\top (y - y_i) \leq J(y_{b(k)})$, for all $i \in \setT_k$, guarantee that $y_{b(k)}$ is not excluded from the search space.
Algorithm~\ref{alg: benders region s-miqp} follows the same flow as described in the first part of the proof of Theorem~\ref{th: algorithm convergence}.
Therefore, we focus on the final case.
Specifically, the solution of $\mathcal{P}_\mathrm{LB-MILP}$ has two possible outcomes: a new $\tilde{y} \neq y_k$ or $\tilde{y} \equiv y_{b(k)}$.
For the former case Algorithm~\ref{alg: benders region s-miqp} proceeds to the next iteration.
For the latter case the algorithm directly terminates since $\LB = \UB$.
\end{proof}

\begin{theorem}
Under Assumptions~\ref{ass: cont and diff fn, integer finiteness}, \ref{ass: constraint qualification}, and \ref{ass: convexity}, Algorithm~\ref{alg: benders region s-miqp} enhanced with the gradient correction and amplification procedure for the Benders cuts, and correction of the infeasibility cuts, \andrea{terminates in a finite number of iterations and returns either a global optimal solution of $\mathcal P_\mathrm{MINLP}$ or a certificate of infeasibility.}
\end{theorem}
\begin{proof}
Assumption~\ref{ass: convexity} makes unnecessary the use of gradient correction and amplification for the Benders cuts.
This follows from Lemma~\ref{lemma: gradient correction for convex MINLP} and Remark~\ref{remark: no amplification for convex MINLP}.
Also, the correction of infeasibility cuts is not required as shown in Lemma~\ref{lemma: correction for infeasibility cuts in convex MINLP}.
Therefore, Algorithm~\ref{alg: benders region s-miqp} is unchanged.
Its convergence follows from Theorem~\ref{th: algorithm convergence}.
\end{proof}

Finally, we aim at presenting an additional result for a class of \acp{MINLP} larger than the convex one.

\begin{assumption}\label{ass: set convexity}
We assume that the set $\Fbar$ is convex.
\end{assumption}

Assumption~\ref{ass: set convexity} includes \acp{MINLP} that have nonconvexities in the space of continuous variables but are convex in the space of integer variables.
A practical example of such a case is a nonlinear control system where the discrete controls enter affinely in the dynamics, \andrea{cf. the system dynamics \eqref{eq: ode unstable ocp} of the example in Sec.~\ref{subsec: unstable ocp}, and real-world applications in~\cite{Buerger2021,Robuschi2021}.}
We can show that Lemma~\ref{lemma: infeasibility convex set} and Lemma~\ref{lemma: correction for infeasibility cuts in convex MINLP} also hold in this case since the cuts are constructed in the integer space only.

\begin{lemma}\label{lemma: correction infeasibility cut for convex F}
Under Assumptions~\ref{ass: cont and diff fn, integer finiteness}, \ref{ass: constraint qualification}, and \ref{ass: set convexity}, if $\prob_\NLP$ is infeasible for $\hat{y}$, and if $\bar{y}$ solves $\mathcal{P}_\FNLP(\hat{y}, y_b)$, then inequality $(\hat{y} - \bar{y})^\T (y - \bar{y}) \leq 0$ is valid for all $y \in \Fbar$.
\end{lemma}
\begin{proof}
The proof follows similarly to Lemma~\ref{lemma: infeasibility convex set}.
\end{proof}

\begin{theorem}\label{theorem: infeasibility convexity}
If Assumptions~\ref{ass: cont and diff fn, integer finiteness}, \ref{ass: constraint qualification}, and \ref{ass: set convexity} hold, then Algorithm~\ref{alg: benders region s-miqp} enhanced with gradient correction and amplification procedure for the Benders cuts, and correction of the infeasibility cuts, \andrea{terminates in a finite number of iterations and returns either a feasible solution of $\mathcal P_\mathrm{MINLP}$ or a certificate of infeasibility.}
\end{theorem}
\begin{proof}
The theorem is trivial if Algorithm~\ref{alg: benders region s-miqp} is initialized with a feasible solution.
This follows from Lemma~\ref{lemma: finiteness of sbmiqp with corrections}.
In the other case, $b(k) \notin \setT_k$, the constraint set for $\mathcal{P}_\mathrm{BR-MIQP}$ is given by
\[
\mathbb{A}_k = \{y \in Y \mid (y_i - \bar{y}_i)^\T (y - \bar{y}_i) \leq 0, \; \forall i \in \{0, \dots, k\} \}.
\]
Lemma~\ref{lemma: correction infeasibility cut for convex F} holds for each $(y_i, \bar{y}_i)$.
Therefore, $\F \subseteq \mathbb{A}_k$.
Every visited infeasible point $y_i$ is excluded from $\mathbb{A}_k$.
Therefore, Algorithm~\ref{alg: benders region s-miqp} either terminates with a certificate of infeasibility since $\F = \emptyset$ or with a feasible point $y_k \in \F$.
\end{proof}

\subsection{Final details and options of Algorithm 1}\label{subsec: implementation details}

\paragraph{Algorithm~\ref{alg: benders region s-miqp} with early termination heuristic.} We introduce a heuristic for Algorithm~\ref{alg: benders region s-miqp} which has shown to be particularly effective in finding high-quality solutions in a short time.
The heuristic consists of a different stopping criterion where the lower bound $\LB$ is set equal to $V_{\mathrm{MIQP}}$, the objective of the MIQP, and it is given by
\begin{equation}
        \UB \geq V_{\mathrm{MIQP}}.
        \label{eq:stop_criteria_LB_V}
\end{equation}
When this heuristic is selected, Algorithm~\ref{alg: benders region s-miqp} does not solve any $\prob_\mathrm{LB-MILP}$ and thus no valid lower bound on the MINLP solution is computed.
This algorithm is available in the \texttt{CAMINO} software package\footnote[1]{\url{https://github.com/minlp-toolbox/CAMINO}} with the name \texttt{s-b-miqp-early-exit}, and is denoted in the following as S-B-MIQP-ee.

\paragraph{Inital guess $y_0$.} We obtain the initial guess for the integer variables $y_0 \in Y$ by solving $\prob_\mathrm{BR-MIQP}$, which is constructed by linearizing around the solution of the relaxed $\prob_\NLP$ computed in line 2 of Algorithm~\ref{alg: benders region s-miqp}.
An issue may arise if the constructed $\prob_\mathrm{BR-MIQP}$ is infeasible, however it has never happened in our computations.
If this happens, it is possible to compute an initial guess using one of the heuristic methods available in \texttt{CAMINO}, e.g., the feasibility pump~\cite{Bertacco2007}.
\texttt{CAMINO} allows concatenating solver calls so that the solution computed by the first solver is used as the initial guess for the next one.
For example, to call the feasibility pump (\texttt{fp}) and then use its solution as the initial guess for S-B-MIQP, it is sufficient to specify the solver \texttt{fp+s-b-miqp} in \texttt{CAMINO}.

\andrea{
        \paragraph{Exact Hessian in $\prob_\mathrm{BR-MIQP}$ for nonconvex MINLPs.}
        When solving \emph{nonconvex} MINLPs, the exact Hessian of the Lagrangian of $\prob_\MINLP$ constructed using primal and dual information of the incumbent solution may become indefinite.
        However, we want the Hessian $B_{b(k)}$ used in $\prob_\mathrm{BR-MIQP}$ to be positive semidefinite such that we always solve \emph{convex} MIQPs.
        In the code implementation, we perform an eigenvalue decomposition of the exact Hessian.
        In case negative eigenvalues are detected, we perform a simple Hessian regularization by adding a diagonal matrix weighted by the absolute value of the smallest eigenvalue, ensuring positive semidefiniteness.
        Furthermore, if the smallest eigenvalue is below a threshold (e.g., $10^{-8}$), we neglect the Hessian and revert to a linear objective.
}

\paragraph{Solution pool.} An important option of \ac{MILP}/\ac{MIQP} solvers that we utilize is the \textit{solution pool}, which corresponds to a set of feasible solutions found during the solution process.
\andrea{
        The user defines the maximum dimension of the solution pool $N_\mathrm{sp}$, i.e., how many integer feasible solutions it should contain.
        Usually, the MIP solver operates in a best-effort manner, storing in the solution pool the best integer solution found during the branch-and-cut procedure.
        Thus, if the solution of the MIP is successful the solution pool contain at least one element and at most $N_\mathrm{sp}$.
        Naturally, we have to solve the corresponding $\prob_\NLP$ for each solution stored in the pool.
        }
By means of the \textit{solution pool}, we obtain multiple cuts per \ac{MIP} solved, making Algorithm~\ref{alg: benders region s-miqp} more efficient.

\paragraph{Additional outer approximation cuts.}
The implementation of Algorithm~\ref{alg: benders region s-miqp} within \texttt{CAMINO} includes an option to add standard outer approximation cuts to the master problems~\cite{Fletcher1994}.
Adding these cuts reduces the computation time in the conducted experiments.
Specifically, given an integer solution $y_k$, the implementation adds the outer approximation cut to lower bound the objective function only when $\prob_\NLP$ is feasible, and the cuts correspond to
\begin{equation}\label{eq: oa cuts on objective}
 f(x_k, y_k) + \nabla f(x_k, y_k)^\top \left(\begin{matrix}
        x - x_k \\ y - y_k
 \end{matrix}\right) \leq 0, \quad k \in \setT_k.
\end{equation}
Moreover, the gradient of \eqref{eq: oa cuts on objective} is corrected and amplified in case the best solution has a better objective, i.e., $J(y_{b(k)}) \leq J(y_k) = f(x_k, y_k)$.
The correction and amplification procedure is similar to the one described in Sec.~\ref{subsec: gradient correction strategy 1}-\ref{subsec: gradient amplification}.
Regarding the outer approximation cuts that overapproximate the feasible set, the implementation only adds cuts corresponding to convex constraints.
In the current implementation, there is a mechanism that automatically detects linear constraints, but generic convex constraints should be labeled by the user when the MINLP is formulated.
In each iteration $k \in \setT_k$, the outer approximation cuts are defined as
\begin{equation}
        g_\mathrm{L}^i(x, y; x_k, y_k) \leq 0, \quad i \in \mathbb{C}_g \subseteq \Z_{[1, n_g]},
\end{equation}
where $\mathbb{C}_g$ contains the indices of inequality constraints where $g^i: \R^{n_x} \times \R^{n_y} \to \R$ is convex in $x, y$ jointly.

\section{Numerical results}\label{sec: results}
In this section, we illustrate the performance of Algorithm~\ref{alg: benders region s-miqp} via numerical experiments.
First, we compare the proposed algorithm against three open-source solvers, Bonmin~\citep{Bonami2005}, SCIP~\citep{Achterberg2009,Bestuzheva2023}, and SHOT~\citep{Lundell2022a, Lundell2022b}, and the commercial solver Gurobi~\citep{Gurobi}, on a large number of \acp{MINLP} selected from the MINLPLib~\citep{Bussieck2003a,MINLPLib}.
\andrea{
        SCIP and Gurobi are global solver for nonconvex MINLPs and they implement a spatial branch-and-bound method.
}
Secondly, we consider two \acp{OCP} for a nonlinear system with binary control input: a textbook example from~\citep[\S 8.17]{Rawlings2017} and a real-world case study of a building climate system from~\citep{Buerger2021}.
Since SCIP, SHOT and Gurobi for MINLPs do not have a direct interface with CasADi, comparing it on these \acp{OCP} is nontrivial.
\andrea{For the simpler textbook example, we reformulated the problem in Pyomo~\citep{hart2011pyomo,bynum2021pyomo} to enable calling SCIP, SHOT and Gurobi.
We did not attempt the same for the more complex climate-system case, because CasADi cannot directly output an AMPL description, and reformulating the problem in Pyomo is not trivial.}
All the presented results are obtained on a computer with a Intel(R) Xeon(R) W-2225 CPU @ 4.10GHz processor with 4 cores and 32 GB of memory.
\andrea{
        Since the benchmark on MINLPLib instances takes several hours, we ran the solvers in parallel, in batch of three, to have one core always free preventing overloading.
        To minimize oscillations in compute power, we disabled CPU boost and set the maximum CPU clock to 4GHz, slightly lower than the nominal CPU clock.
}

\subsection{MINLPLib instances}\label{sec: comparison benchmark}
This section compares Algorithm~\ref{alg: benders region s-miqp} (S-B-MIQP) and its variant with early termination (S-B-MIQP-ee) with four existing solvers, SHOT v1.1~\citep{Lundell2022a, Lundell2022b}, Bonmin v1.8~\citep{Bonami2005}, SCIP v9.2.2~\citep{Achterberg2009,Bestuzheva2023}, Gurobi v13.0.0~\citep{Gurobi} on a subset of instances from MINLPLib~\citep{Bussieck2003a,MINLPLib}.
The problem instances are selected according to the following requirements: 1) mixed-integer and mixed-boolean variables with nonlinear constraints and/or objective, \andrea{2) the instances have a nl-file representation compatible with CasADi nl-reader}.
Based on these criteria, \andrea{233} convex MINLP instances and \andrea{263} nonconvex MINLP instances were selected.
\andrea{Point 2) excludes instances using AMPL suffices and some other special syntax, only 7 instances are excluded by this criterion.}
S-B-MIQP, S-B-MIQP-ee, Bonmin are available in the \texttt{CAMINO} software package\footnote[1]{\url{https://github.com/minlp-toolbox/CAMINO}} with the name \texttt{s-b-miqp}, \texttt{s-b-miqp-early-termination} and \texttt{bonmin}, respectively.
\andrea{
        SHOT is called from command line interface, load the MINLPLib instance as a \texttt{.nl} file and solve it.
        SCIP and Gurobi (for MINLPs) are called using the AMPL Python interface which loads the MINLPLib as \texttt{.mod} file.}
For reproducibility the results presented in this section are available in the public repository \texttt{CAMINO-benchmark}\footnote[2]{\url{https://github.com/minlp-toolbox/CAMINO-benchmark}}.

S-B-MIQP, Bonmin and SHOT are configured to use the same NLP-solver, Ipopt 3.14 \citep{Waechter2006} with ma27 \citep{HSL} to have a fair comparison.
The \ac{MIP} subproblems are solved using Gurobi 13.0.1~\citep{Gurobi} on a single thread with a solution pool of dimension 10 for SHOT and 5 for both S-B-MIQP and S-B-MIQP-ee.
\andrea{SCIP and Gurobi utilize their algorithms for MINLPs.}
For each problem, the maximum wall time is 300 seconds.
When the time is over we return the feasible solution with best objective if available, otherwise we consider it as a failure.
The primal solution satisfies a tolerance of $10^{-8}$ on both the objective and constraints.
The MINLP gap for termination is set to $10^{-2}$.
The nl-file of each problem contains an initial guess.
\andrea{
        For both S-B-MIQP algorithms, the initial point $y_0 \in Y$ is generated using the strategy described in Sec.~\ref{subsec: implementation details}.
        Specifically, we take the initial point provided in the nl-file and use it as a warm start for IPOPT to solve the integer relaxation of $\prob_\MINLP$.
        The solution of this relaxation is then used to construct the first $\prob_\mathrm{BR-MIQP}$, whose solution yields $y_0$.
}
The parameters of both S-B-MIQP algorithms are set as follows: the hyper-parameter $\alpha=0.5$, the gradient amplification $\rho=1.5$, the weight matrix $W$ is the identity matrix, and $B$ is the exact Hessian in $\prob_\mathrm{BR-MIQP}$.
\andrea{
        The exact Hessian is constructed by evaluating the second-order derivatives of the cost and nonlinear constraints of $\prob_\MINLP$ at the best solution found, which also include the optimal multipliers corresponding to the nonlinear constraints.
        At the end of this subsection we present a sensitivity analysis for tuning $\alpha$ and $\rho$.
}

We compare the algorithms using performance profiles for the objective value of the solutions and for the wall time.
The wall time is computed in the same way for each algorithm, by starting a clock when the respective algorithm is called and by stopping the clock when the termination condition is met.
For the performance profile of the wall time we used the standard method described in \cite{Dolan2002}.
Consider $q_{p,s}$ the quantity of interest required to solve problem $p \in \mathcal{P}$ by solver $s \in \mathcal{S}$, where $\mathcal{P}$ is the set of problems and $\mathcal{S}$ is the set of solvers.
Then, the performance ratio is defined as
\begin{equation}
        r_{p, s} = \frac{q_{p, s}}{\min\{q_{p, s} \; | \; s \in \mathcal{S}\}},
\end{equation}
in case solver $s$ cannot solve problem $p$, the ratio $r_{p, s} = \infty$.
The performance profile is defined as
\begin{equation}
        \pi_s(\tau) = \frac{\mathrm{card}(p \in \mathcal{P} \; | \; r_{p, s} \leq \tau)}{\mathrm{card}(\mathcal{P})},
\end{equation}
where $\mathrm{card}$ denotes the cardinality of a set.
Thus, $\pi_s(\tau)$ is the probability for solver $s \in \mathcal{S}$ that a performance $r_{p, s}$ is within factor $\tau > 0$ of the best possible ratio.
The function $\pi_s$ is the cumulative distribution function of the performance ratio.

Since the objective value of a problem $p\in \mathcal{P}$ can be negative, in order to compute the performance profile we apply a scaling to obtain only positive numbers,
\begin{equation*}
        \text{if } \; q^* = \min\{q_{p, s} \; | \; s \in \mathcal{S}\} < 0 \quad \text{then } \; \hat{q}_{p, s} = q_{p, s} - q^* + 1, \quad s \in \mathcal{S}.
\end{equation*}

\andrea{
        For the 233 convex MINLPs, we summarize the results in Table \ref{tab: convex minlp stats}.
        We define as ``success'' the instances that reached the desired MINLP gap within time limit, as ``failures'' the instances where a feasible solution is not available at the time limit, and as ``time-out'' the instances where a feasible solution is available but the desired gap is not yet achieved.
        For the latter case, we report the number of instances whose gap is smaller than 0.1, an interval 10 times larger than the sought MINLP gap.
        We report S-B-MIQP-ee at the bottom of Table~\ref{tab: convex minlp stats} since it is a heuristic method even for convex MINLPs, so the status ``success'' does not necessarily correspond to the global minimizer found rather to a feasible solution found.
}
\begin{table}[h!]
        \centering
        \begin{tabular}{l|rrrr}
        \toprule
        Solver & Success & Fail & Time-out & Gap $<10^{-1}$ \\
        \midrule
        Bonmin       & 128 & 66 & 39 & 27  \\
        Gurobi       & 202 & 2 & 29 & 23 \\
        SCIP         & 185 & 2 & 46 & 38 \\
        SHOT         & 212 & 9 & 12 & 10  \\
        S-B-MIQP     & 200 & 7 & 26 & 19 \\
        \midrule
        S-B-MIQP-ee  & 224 & 8 & 1 & 1  \\
        \bottomrule
        \end{tabular}
        \caption{Summary of benchmark for convex MINLPs. \label{tab: convex minlp stats}}
\end{table}

Figure~\ref{fig:benchmark_comparison_conv} depicts the performance profiles of objective value and wall time achieved by the different algorithms on the set of convex MINLPs.
\andrea{
        Regarding wall time, Gurobi and SHOT are the fastest solvers that guarantee to find global optimal solutions.
        S-B-MIQP is about 4 times slower than SHOT.
        S-B-MIQP and SCIP have similar wall-time profiles, but we can claim that S-B-MIQP is slightly after since it achieves 202 success within the available time while SCIP achieves 185 success.
        Nevertheless, SCIP is a bit more robust as it fails only for 2 instances (\texttt{fac2} and \texttt{tls12}) while S-B-MIQP fails for 7 instances (\texttt{ibs2}, \texttt{o7}, \texttt{p\_ball\_30b\_10p\_2d\_h}, \texttt{tls12}, \texttt{tls5}, \texttt{tls6}, \texttt{tls7}).
        SHOT fails for 9 instances, namely \texttt{ibs2} and the \texttt{p\_ball} problems \texttt{10b\_7p\_3d\_h}, \texttt{20b\_5p\_2d\_h}, \texttt{30b\_10p\_2d\_h}, \texttt{30b\_5p\_2d\_h}, \texttt{30b\_5p\_3d\_h}, \texttt{30b\_7p\_2d\_h}, \texttt{40b\_5p\_3d\_h}, and \texttt{40b\_5p\_4d\_h}.
        Gurobi fails for \texttt{p\_ball\_40b\_5p\_4d\_h} and \texttt{tls12}.
        Bonmin solves less problems overall and needs more time for doing it compared to the other algorithms.
        The heuristic algorithm S-B-MIQP-ee is of course faster than S-B-MIQP as it does not solve any $\prob_\mathrm{LB-MILP}$, and on many instances is the second or third fastest algorithm.
        Moreover, S-B-MIQP-ee finds globally optimal solution for about 70\% of the problems (cf. left plot of Figure~\ref{fig:benchmark_comparison_conv}).
        However, this cannot be proven in general.
        Overall SHOT is the fastest open-source algorithm; yet, in the next subsection we show that, once the Python overhead is removed, S-B-MIQP can achieve similar performance.
        }

\andrea{
        We comment on the instances where S-B-MIQP fails.
        Instance \texttt{ibs2} is a relatively large-size problem with few nonlinear constraints having a dense Jacobian, we believe that there is some issue with its nl-file representation, as it fails for all solvers reading nl-files, while it solved successfully using SCIP and Gurobi which read mod-files.
        For \texttt{o7} the solution of $\prob_\mathrm{BR-MIQP}$ is quite slow as the MIQP solver struggles to improve lower bounds.
        By implementing a watchdog on the gap of the MIQP solver, we could abort the computation and return the best feasible solution on the tree, without waiting that the tolerance on the gap of $\prob_\mathrm{BR-MIQP}$ is met.
        In this way S-B-MIQP can proceed, generate more cuts and find the solution.
        We tested this watchdog strategy and we successfully solved these two problems.
        However, the strategy is not included in the main release of \texttt{s-b-miqp}.
        For the instance \texttt{p\_ball\_30b\_10p\_2d\_h}, S-B-MIQP performs six iterations in the given time budget but it in unable to compute a feasible solution.
        The problem might be particularly hard since MINLPLib does not report a valid dual bound.
        For instances \texttt{tls}, S-B-MIQP is stuck in a loop that solves feasibility NLP and add feasibility cuts to $\prob_\mathrm{BR-MIQP}$.
        The infeasibility cuts seem to work properly, since the objective of $\prob_\FNLP$ is decreasing during the iterations.
        We suspect numerical issues in the solution of $\prob_\NLP$, also because once the integer variables are fixed, $\prob_\NLP$ is overconstrained, having more equality constraints than optimization variables.
        We believe that all these failures could be mitigated by implementing some reformulation routines and bound tightening for the MINLP, as done for instance in SHOT.
        Table \ref{tab: appendix convex problems} in the Appendix reports the objective and wall time achieved by each algorithm.
}
\begin{figure}
        \centering
        \includegraphics[width=0.48\textwidth]{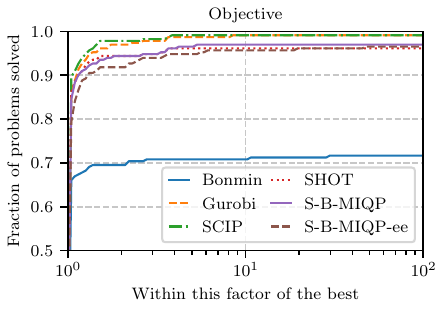}
        \includegraphics[width=0.48\textwidth]{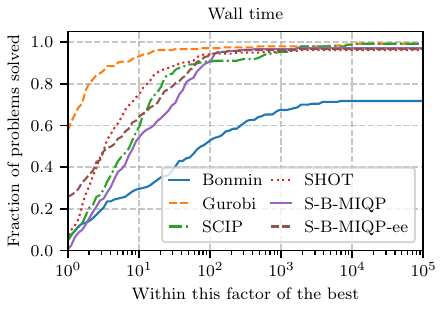}
        \caption{Convex MINLPs. Comparison of Algorithm~\ref{alg: benders region s-miqp} with standard termination (S-B-MIQP) and with early termination (S-B-MIQP-ee), cf. Sec.~\ref{subsec: implementation details}, against Bonmin, Gurobi, SCIP, and SHOT.
        Left: Performance profiles for the objective value achieved by each solver.
        Right: Performance profiles for the wall time of each solver.
        }
        \label{fig:benchmark_comparison_conv}
\end{figure}

\andrea{
        We summarize the results for nonconvex MINLPs in Table \ref{tab: nonconvex minlp stats}.
        Differently from above, all algorithms except SCIP and Gurobi are only heuristics for nonconvex MINLPs.
        Therefore, we classify as ``success'' a feasible solution produced by solver when its termination is triggered.
        We classify as ``fail'' any instance where the solver cannot find a feasible solution within the time limit.
        Finally, the instances counted as ``time-out'' correspond the best feasible solution found by the solver when the time limit is reached.
        Since SCIP and Gurobi are global solver comparing their solution time against the other solver would not be meaningful as we expect that the spatial branch-and-bound procedure they execute to compute valid lower bounds is much more expensive than simply find feasible solutions.
        Therefore, on each instance we limit the available time of SCIP and Gurobi to the computation time achieved by S-B-MIQP.
        In case S-B-MIQP failed for a specific instance we consider the standard time limit of 300 seconds.
        Doing so we obtain a comparison among primal heuristics.
        We remark that the termination conditions of S-B-MIQP, SHOT and Bonmin are unchanged, however for nonconvex MINLPs the computed lower bounds might be wrong.
        Thus, the corresponding termination conditions might be triggered for feasible solutions that are not global minimizers.
}

\begin{table}[h!]
        \centering
        \begin{tabular}{l|rrrr}
                \toprule
                Solver & Success & Fail & Time-out & Gap $<10^{-1}$ \\
                \midrule
                Bonmin      & 133 & 121 & 15  & -  \\
                SHOT        & 128 & 133 & 8  & -  \\
                S-B-MIQP    & 177 & 82 & 10  & -  \\
                S-B-MIQP-ee & 176 & 89 & 4  & -  \\
                \midrule
                Gurobi      & 170 & 73 & 26 & 21 \\
                SCIP        & 155 & 85 & 29 & 22 \\
                \bottomrule
        \end{tabular}
        \caption{Summary of benchmark for nonconvex MINLPs. For SCIP and Gurobi ``success'' correspond to a global optimal solution, for the other solvers not. For SCIP and Gurobi the time limit is set equal to the computation time achieved by S-B-MIQP on the same instance.\label{tab: nonconvex minlp stats}}
\end{table}

\andrea{
        For nonconvex MINLPs, Gurobi solves the largest share of problems, closely followed by S-B-MIQP and SCIP.
        Bonmin and SHOT close the ranking with a percentage of failures close to 50\%.
        }
Figure~\ref{fig:benchmark_comparison_nonconv} depicts the performance profiles of objective value and wall time achieved by the different algorithms on the set of nonconvex MINLPs.
\andrea{
Regarding the objective value, S-B-MIQP, Gurobi, and SCIP find the same solutions for about 40\% of the problems.
For the remaining 30\%, SCIP is slightly dominated.
Bonmin and SHOT perform well only for 25\% of the problems, on the remaining problems they yield worse solutions than the other solvers.
}
\andrea{
Regarding the wall time, Gurobi is the fastest solver.
Despite having set its time limit identical to the one S-B-MIQP, for some instances where S-B-MIQP reach the time limit Gurobi converges to the global minimizer in less time.
Similarly, in some cases Gurobi can find feasible solutions while the other solver fails.
Conversely, the wall-time profile of SCIP is very much overlapped with the one of S-B-MIQP as they share the same time limit.
S-B-MIQP-ee is only marginally better than S-B-MIQP, showing that for the majority of the problem, when $\prob_\mathrm{LB-MILP}$ is solved, it returns a solution with $V_\mathrm{MILP} \ge \UB$, triggering immediate termination of S-B-MIQP.
The wall-time profiles of Bonmin and SHOT are completely dominated by the other solvers.
}
SHOT is particularly slow, suggesting that its algorithm might not be a good heuristic for nonconvex MINLPs.
\andrea{Table \ref{tab: appendix nonconvex problems} in the Appendix reports the objective and wall time achieved by each algorithm.}

\begin{figure}
        \centering
        \includegraphics[width=0.48\textwidth]{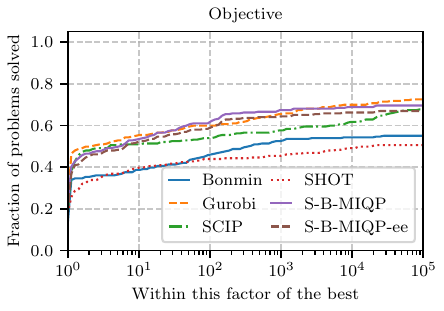}
        \includegraphics[width=0.48\textwidth]{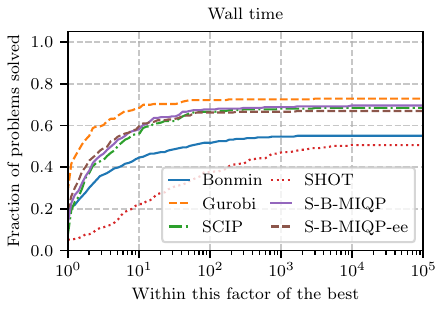}
        \caption{Nonconvex MINLPs. Comparison of Algorithm~\ref{alg: benders region s-miqp} with standard termination (S-B-MIQP) and with early termination (S-B-MIQP-ee), cf. Sec.~\ref{subsec: implementation details}, against Bonmin, Gurobi, SCIP, and SHOT.
        Left: Performance profiles for the objective value achieved by each solver.
        Right: Performance profiles for the wall time of each solver.
        }
        \label{fig:benchmark_comparison_nonconv}
\end{figure}

\subsubsection{Considering only solver time for S-B-MIQP algorithms}
\andrea{
To illustrate the algorithmic potential of S-B-MIQP for solving convex MINLPs, we evaluate the solver time independently of the Python framework overhead.
Recall that \texttt{CAMINO} implements S-B-MIQP using Python for high-level operations, while lower-level solver interfacing relies on CasADi.
In contrast, solvers like SHOT or SCIP are implemented in compiled code and do not incur this Python overhead.
Therefore, in Fig.~\ref{fig:solve time only}, we present a performance profile considering only the wall time spent solving the subproblems; this excludes operations performed in Python, such as cut generation and management.
Comparing Fig.~\ref{fig:solve time only} with the right plot in Fig.~\ref{fig:benchmark_comparison_conv}, we observe that S-B-MIQP performs comparably to SHOT and outperforms SCIP.
Furthermore, the S-B-MIQP-ee heuristic emerges as the second-fastest solver.
These results demonstrate that a compiled implementation of S-B-MIQP would achieve state-of-the-art performance.
Crucially, the algorithm currently operates without the advanced presolving routines standard in SHOT, such as bound tightening, model reformulation, and constraint disaggregation, suggesting that its performance could even improve.
}

\begin{figure}
        \centering
        \includegraphics{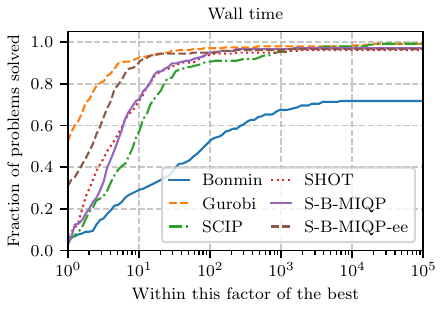}
        \caption{Performance profile for wall time obtained on convex MINLPs. Differently from the right plot of Fig.~\ref{fig:benchmark_comparison_conv}, for S-B-MIQP and S-B-MIQP-ee we considered as part of the wall time only the time spent into the subsolvers, disregarding the time spent in Python operations.
        This figure evidences that the algorithmic logic of S-B-MIQP competes with SHOT, suggesting that an implementation in compiled code would be state-of-the-art.
        }
        \label{fig:solve time only}
\end{figure}

\subsubsection{Sensitivity analysis for tuning $\alpha$ and $\rho$}
\andrea{
        Since the hyper-parameter $\alpha$ used in \eqref{eq: reduced rhs J_bar} enters only the computation of the Benders region in $\prob_\mathrm{BR-MIQP}$, we tune its value by using the algorithm S-B-MIQP-ee.
        If we would use S-B-MIQP, the benefit of choosing different $\alpha$ is less evident since the overall algorithm wall-time is influenced also by the solution of $\prob_\mathrm{LB-MILP}$.
        We compared S-B-MIQP-ee with $\alpha \in \{0.05, 0.25, 0.5, 0.75, 0.95\}$ on convex MINLP instances, reporting performance profiles for both objective value of the solution and wall time.
        As shown in the left plot of Fig.~\ref{fig:tuning alpha}, S-B-MIQP-ee is largely insensitive to $\alpha$ with respect to the objective value, with differences emerging on only about 4\% of the problems.
        Conversely, the choice of $\alpha$ has a strong impact on wall time (see the right plot of Fig.~\ref{fig:tuning alpha}), where lower values of $\alpha$ dominate the performance profile.
        These results align with our expectations: an $\alpha$ close to zero produces a smaller Benders region, forcing $\prob_\mathrm{BR-MIQP}$ to find a solution that is lower bounded by all Benders cuts while having an objective smaller than the reduced right-hand side, cf.~\eqref{eq: reduced rhs J_bar}.
        With small $\alpha$ is more likely than the feasible set of $\prob_\mathrm{BR-MIQP}$ becomes empty, thus triggering the termination of S-B-MIQP-ee.
        We selected $\alpha=0.5$ as the default value for S-B-MIQP algorithms.
        Although $\alpha=0.05$ and $\alpha=0.25$ are marginally faster on average, they failed to solve the \texttt{gams01} instance within the time limit.
}

\begin{figure}
        \centering
        \includegraphics[width=0.48\textwidth]{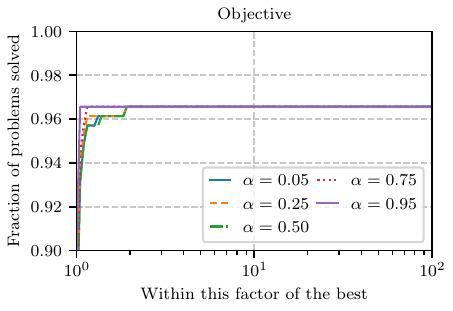}
        \includegraphics[width=0.48\textwidth]{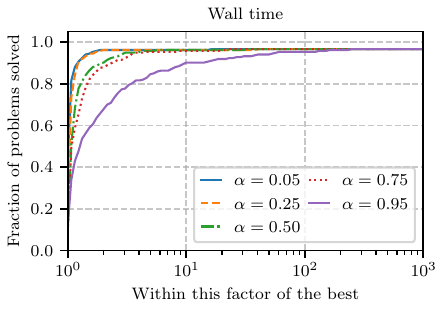}
        \caption{Sensitivity analysis of S-B-MIQP-ee to the hyper-parameter $\alpha$ used to shrink the Benders region of $\prob_\mathrm{BR-MIQP}$ according to \eqref{eq: reduced rhs J_bar}.
        Note the different scales of the axes.
        The considered problems are the 233 convex MINLPs instances selected from MINLPLib.
        }
        \label{fig:tuning alpha}
\end{figure}

\andrea{
        To tune the hyperparameter $\rho$ used to amplify the gradient of the corrected cuts according to \eqref{eq: grad expansion}, we focused on the 187 nonconvex instances where the cut correction routine of S-B-MIQP was invoked.
        Fig.~\ref{fig:tuning rho} depicts the performance profile for objective value and wall time obtained by running S-B-MIQP with $\rho \in \{1, 1.5, 5, 10, 50\}$.
        In terms of objective value, the algorithm converges to the same solution for approximately 98\% of the 187 problems.
        The variant with $\rho=1.5$ achieves slightly better performance.
        Notably, the instance \texttt{mbtd} reaches the time limit for all $\rho > 1.5$.
        Concerning wall time, the plot is dominated by $\rho=1$, which is expected, since without gradient amplification S-B-MIQP terminates as soon as a feasible solution is found.
        Overall, the values $\rho \in \{1.5, 5, 50\}$ perform very similarly, while $\rho=10$ results in slightly longer wall times.
        We selected $\rho=1.5$ as the default value for S-B-MIQP algorithms, as it achieves the best solution quality while remaining competitive in computation time.
}

\begin{figure}
        \centering
        \includegraphics[width=0.48\textwidth]{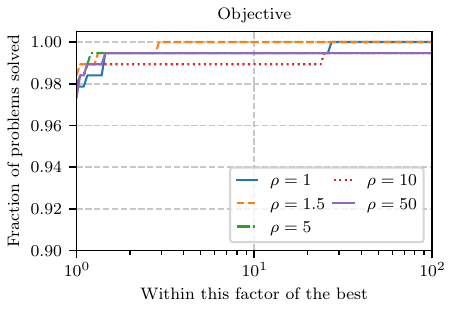}
        \includegraphics[width=0.48\textwidth]{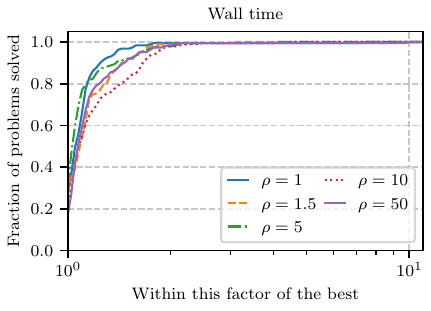}
        \caption{Sensitivity analysis of S-B-MIQP to the hyper-parameter $\rho$ used to amplify the gradient of the corrected cuts according to \eqref{eq: grad expansion}.
        Note the different scales of the axes.
        The considered problems in the analysis are 187 instances among the nonconvex MINLPs where the cut correction routine \eqref{op: minimal gradient correction} is called.
        }
        \label{fig:tuning rho}
\end{figure}

\subsection{MIOC of an unstable nonlinear system}\label{subsec: unstable ocp}
Consider a reference tracking task for a nonlinear unstable system with discrete control input taken from~\citep[\S 8.17]{Rawlings2017}.
The dynamic is described by the \ac{ODE}
\begin{equation}\label{eq: ode unstable ocp}
        \dot{x}(t) = x^3(t) - u(t), \quad t \in [t_0, t_\mathrm{f}]
\end{equation}
where the state is denoted by $x(t) \in \R$ and the control by $u(t) \in \{0, 1\}$.
The control is subject to a minimum dwell time constraint of $0.1\, \mathrm{s}$.
The aim is to track the state reference $x_\mathrm{ref} = 0.7$, starting from the initial state $\bar{x}_0 = 0.9$.
By means of the multiple shooting approach for direct optimal control \citep{Bock1984}, we discretize the \ac{ODE} over a fixed grid with $N=30$ intervals such that $t_0 < t_1 < ... < t_N = t_\mathrm{f}$ adopting a 4th order explicit Runge-Kutta (RK) integrator and a sampling time $t_\mathrm{s}=0.05 \, \mathrm{s}$.
The resulting discretized \ac{OCP} is
\begin{mini}[2]
        {\substack{x_0, u_0, \dots, \\ u_{N-1}, x_N}}{\sum_{k=0}^N (x_k - x_\mathrm{ref})^2}{}{\label{op: unstable ocp}}
        \addConstraint{x_0 =}{\bar{x}_0}
        \addConstraint{x_{k+1}=}{F_\mathrm{RK}(x_k, u_k),}{\quad k=0, \dots, N-1}
        \addConstraint{u_k \in}{\mathcal{U}}{\quad k=0, \dots, N-1,}
\end{mini}
where $\mathcal{U}$ $\coloneqq \{u \in \{0,1\}^{N-1} \; \vert \; u_k \geq u_{k-1} - u_{k-2},\; k=0, \dots, N-1\}$ imposes a minimum uptime for the control inputs of two consecutive time steps.
The required previous values $u_{-1}, u_{-2}$ are set to zero.
Moreover, $x_\mathrm{ref}$ denotes the state reference to track and function $F_\mathrm{RK}$ corresponds to the RK integrator.
Problem \eqref{op: unstable ocp} is a nonconvex MINLP which we solved using different algorithms.
The results are listed in Table~\ref{tab: unstable ocp results}, and Figure~\ref{fig: trajectory unstable ocp} depicts the globally optimal state and control trajectories of~\eqref{op: unstable ocp}.
\andrea{
In Table~\ref{tab: unstable ocp results} we divided the solvers in two sets.
In fact, for Bonmin, \ac{CIA}~\citep{Sager2011a} and S-B-MIQP with Gauss-Newton (GN) Hessian approximation we formulated the problem directly in \texttt{CAMINO} using CasADi.
The native formulation in \texttt{CAMINO} allows us to have more control over the algorithms, for instance, we specified the Hessian approximation to use in $\prob_\mathrm{BR-MIQP}$ of S-B-MIQP.
Also, we can flag dwell-time constraints such that they are dropped from $\prob_\NLP$ of S-B-MIQP.
Doing this improves the solution of the relaxed $\prob_\MINLP$ used to compute the starting point $y_0 \in Y$.
The other solvers, Gurobi, SCIP, and SHOT, are called via AMPL and they load an nl-file representation of \eqref{op: unstable ocp} generated via Pyomo, the model is open-source\footnote[1]{\url{https://github.com/minlp-toolbox/misc/blob/main/pyomo_unstable_ocp.py}}.
We also solve the nl-file representation of \eqref{op: unstable ocp} via S-B-MIQP to show the performance of the proposed algorithm when it is used as a ``black box''.
Indeed, the nl-file does not allow us to specify the same additional information in the model formulation.
Therefore, in this case the exact Hessian is used in $\prob_\mathrm{BR-MIQP}$ and dwell-time constraints will be part of each $\prob_\NLP$.
}

The MINLP gap is set to $10^{-4}$, solvers share the same parameters when possible and multithreading is allowed.
Bonmin ran with its default nonlinear branch-and-bound routine.
S-B-MIQP utilized Gurobi as solver for its MILP/MIQP master problems with solution pool dimension equal 5.
The \ac{CIA} master problem was solved with the tailored branch-and-bound solver pycombina~\citep{Buerger2020a}.
Finally, the two global solvers, SCIP and Gurobi ran their default spatial branch-and-bound algorithm.
For further implementation details, we refer the reader to the collection of problems in the \texttt{CAMINO} repository.
We noticed that both Bonmin and S-B-MIQP with GN Hessian found the global optimum reported in~\citep[\S 8.17]{Rawlings2017}, while the specialized algorithm \ac{CIA} returned a slightly suboptimal solution but in only 0.029 seconds.
Interestingly, S-B-MIQP with GN Hessian returns the solution in only 0.35 seconds while Bonmin takes more than 10 seconds.
The two global solvers SCIP and Gurobi find the global optimum taking few seconds.
Similar time is required by SHOT and S-B-MIQP with exact Hessian but they found suboptimal solutions even worse than CIA.
Bonmin's nonlinear branch-and-bound was the slowest method, taking more than 10 seconds.

\begin{figure}[h!]
        \centering
        \includegraphics[width=0.5\columnwidth]{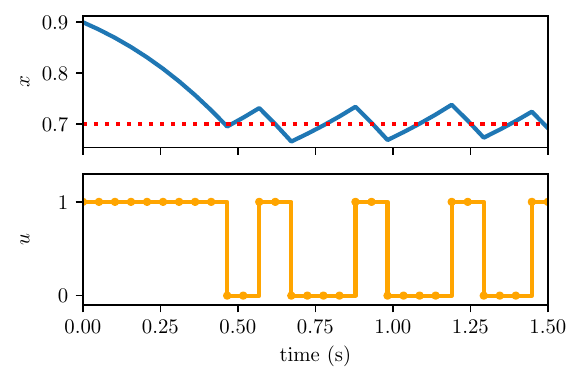}
        \caption{Optimal state and control trajectories of \eqref{op: unstable ocp}.}
        \label{fig: trajectory unstable ocp}
\end{figure}
\begin{table}[h!]
        \centering
        \begin{tabular}{l|rr} \toprule
                Algorithm & Objective & Runtime [s] \\ \midrule
                \multicolumn{3}{l}{\textit{Natively in CAMINO}} \\
                Bonmin & 0.1765 & 11.43 \\
                CIA~\citep{Buerger2021} & 0.1771 & 0.029 \\
                S-B-MIQP (w. GN Hessian) & 0.1765 & 0.35 \\
                \midrule
                \multicolumn{3}{l}{\textit{Loading .nl file}} \\
                Gurobi$^*$ & 0.1765 & 6.16 \\
                SCIP$^*$ & 0.1765 & 3.92 \\
                SHOT & 0.1775 & 3.43 \\
                S-B-MIQP (w. exact Hessian) & 0.1792 & 2.58 \\
                \bottomrule
                \multicolumn{3}{l}{\multirow{2}{*}{\parbox{5cm}{$^*$Global solvers.}}}
        \end{tabular}
        \caption{Comparison of different algorithms for the solution of \eqref{op: unstable ocp}. \label{tab: unstable ocp results}}
\end{table}

\subsection{MIOC of a renewable energy system}
To assess the performance of Algorithm \ref{alg: benders region s-miqp} on an example of real-world complexity, we consider an extended version of the solar-thermal-climate-system (STCS) described in \citep{Buerger2021} and physically installed at Karlsruhe University of Applied Sciences.
The system provides cooling for the building's main hall by means of thermal machines driven by solar-thermal energy.
Specifically, the system consists of an \ac{ACM} and a \ac{HP}.
The first machine can be used in \ac{AC} mode, during which the solar thermal heat drives the machine ($b_\mathrm{ac}=1$), or in \ac{FC} ($b_\mathrm{fc}=1$), where the cooling tower installed on the roof of the building can directly cool down the medium at ambient temperature.
The second machine is a regular \ac{HP} that has been installed more recently to provide flexibility and extra cooling power.
When it is switched on ($b_\mathrm{hp}=1$), it immediately provides cooling energy since the machine is connected to the low-temperature storage and driven by electric energy.
The electric energy required to drive the system is either bought from the grid or supplied by on-site PV panels.
A schematic of the current plant is contained in~\citep{Ghezzi2023a}.
Experimental operations and numerical case studies have been carried out on this system in \citep{Buerger2019, Buerger2020, Buerger2021, Buerger2023, Ghezzi2023a}.

The system dynamics are modeled via a set of \acp{ODE} with $\xi \in \R^{19}$ differential states, $\mu \in \R^6$ continuous controls and $\nu \in \{0, 1\}^3$ binary controls.
The system state includes the temperature of the flat plate and vacuum tube solar collectors, $T_\mathrm{fpsc}$ and $T_\mathrm{vtsc}$ respectively, four temperature levels in the stratified high-temperature storage, $T_\mathrm{ht, i}, i=1, 2, 3, 4$, the temperature of the low-temperature storage $T_\mathrm{lt}$, and the temperature inside the primary and secondary solar circuit, $T_\mathrm{psc}$ and $T_\mathrm{ssc}$, respectively.
The continuous controls regulate the electric power absorbed from the grid, velocity and pressure of the pumps in the solar circuit, and the input/output flow rate of the high-temperature storage.
As described above, there are three binary controls that decide the switching on/off of the different machines.
The system is subject to several ambient conditions, represented as time-varying parameters in the \ac{NLP}.
These are the ambient temperature, the solar irradiance on the solar collectors, the power generated by the local PV panels, the price of electric energy, and the desired cooling load profile.

It is now possible to formulate a \ac{MIOCP} that aims to provide the specified cooling power while operating the system safely and energy-efficiently.
Using a direct approach, we discretize the \ac{MIOCP}, which has a time horizon of 24~hours, via Gauss-Radau collocation of order~3.
We divide the time horizon into $N = 24$ intervals with a sampling time of 1~hour.
To guarantee feasibility of the \ac{MIOCP}, we introduce $n_\mathrm{s} = 24$ slack variables in each discretization interval, which are penalized linearly and quadratically in the cost function.
Therefore, the STCS problem has $N \cdot (19 \cdot d + 6 + 3 + n_\mathrm{s}) = 2160$ variables, of which $3 \cdot N = 72$ are binary.

By a stage-wise concatenation of the variables, we define the vector of continuous and binary decision variables as $x$ and $y$, respectively.
Thus, the resulting \ac{MINLP} can be stated compactly as
\begin{mini}
        {\substack{x \in \R^{n_x}, \\y \in \{0, 1\}^{n_y}}}{\norm{f_1(x, y)}^2 + f_2(x, y)}{}{\label{op: stcs ocp}}
        \addConstraint{g(x, y)}{\leq 0}
        \addConstraint{h(x, y)}{= 0,}
\end{mini}
where the cost function is the sum of a quadratic term, aiming to minimize constraint violation and achieve smooth actuation of the mixing valves, and a nonlinear one, defined by $f_2$, which incorporates the electricity cost for operating the system.
The special structure of the cost function allows for positive semidefinite Hessians for the \acp{MIQP} via the Gauss-Newton approximation \citep{Messerer2021a}.

The \ac{MIOCP} is modeled and solved using \texttt{CAMINO}.
We compare the solutions obtained with S-B-MIQP, S-B-MIQP-ee, the specialized algorithm CIA~\citep{Sager2011a,Buerger2020a}, and the Bonmin nonlinear branch-and-bound method.
The algorithms share the same parameters where applicable.
The MINLP gap for termination is set to $10^{-2}$ and the computation time limit is set to 30~minutes for each algorithm.
IPOPT~\citep{Waechter2006} uses HSL ma57~\citep{HSL} as the internal linear solver.
The \ac{MIP} subproblems are solved using Gurobi with multithreading enabled and a solution pool of dimension 3.
Moreover, the MIP gap of $\prob_\mathrm{BR\text{-}MIQP}$ is set to 10\%, while that of $\prob_\mathrm{LB\text{-}MILP}$ is set to 5\%.
The time limit for each master problem in every S-B-MIQP iteration is the maximum between 600~seconds and the remaining overall MINLP time budget.
If the master problems hit the time limit, we accept the available incumbent solution, even if it has a MIP gap larger than prescribed.
If no solution is available, the corresponding master problem is considered a failure.
If $\prob_\MILP$ cannot be solved, S-B-MIQP terminates and returns the best solution found.
The two hyper-parameters of S-B-MIQP, $\alpha$~\eqref{eq: reduced rhs J_bar} and $\rho$~\eqref{eq: grad expansion}, are set to their default values of $0.5$ and $1.5$, respectively.
For more details regarding the model and solver options, we invite the reader to consult the published code\footnote[1]{\url{https://github.com/minlp-toolbox/CAMINO/blob/main/camino/problems/solarsys/__init__.py}}.

We report the objective values of the solutions and the runtimes in Table~\ref{tab: stcs ocp results}.
\begin{table}
        \centering
        \ra{1.2}
        \begin{tabular}{l|rrr} \toprule
                Algorithm & Objective & Runtime (mm:ss) & Timestamp Best Solution \\ \midrule
                \textit{Relaxed} NLP            & 1547.33 & 00:05 &  -               \\
                Bonmin                          & 2409.66 & 30:00 &  02:43        \\
                CIA \citep{Buerger2021}         & 6370.03 & 00:08 &  -               \\
                S-B-MIQP                        & 2288.40 & 30:00  & 01:28 \\
                S-B-MIQP-ee                     & 2288.40 & 30:00  & 01:28 \\
                \bottomrule
        \end{tabular}
        \caption{Comparison of different algorithms for the solution of STCS~\eqref{op: stcs ocp}. \label{tab: stcs ocp results}}
\end{table}
For the \ac{CIA} algorithm, we quickly obtain a solution but the corresponding objective is fairly high, corresponding to a very suboptimal operating strategy for the energy system.
For the other algorithms, the termination is triggered by reaching the time limit of 30 minutes.
However, in each case the best solution is found much earlier.
Bonmin finds its best solution in less than 3 minutes while S-B-MIQP takes less than 2 minutes.
Differently from CIA, both Bonmin and S-B-MIQP deliver solutions that result in a high-performing operation of the plant.
In Figure~\ref{fig: trajectory stcs}, we plot the optimal trajectories of the binary controls and of a selected subset of the state extracted from the best solution computed by S-B-MIQP.
The plots show an almost optimal operation of the system, where temperature bounds and predictions of the external parameters over the horizon have been exploited by the optimizer.
However, since constraints are imposed in a soft way, the solution exhibits a small violation of the upper bound of the low-temperature storage.
Note that the free-cooling mode is never active, meaning that the system never dissipates heat into the environment rather storing it in the high temperature water tank.
\begin{figure}
        \centering
        \includegraphics[width=\columnwidth]{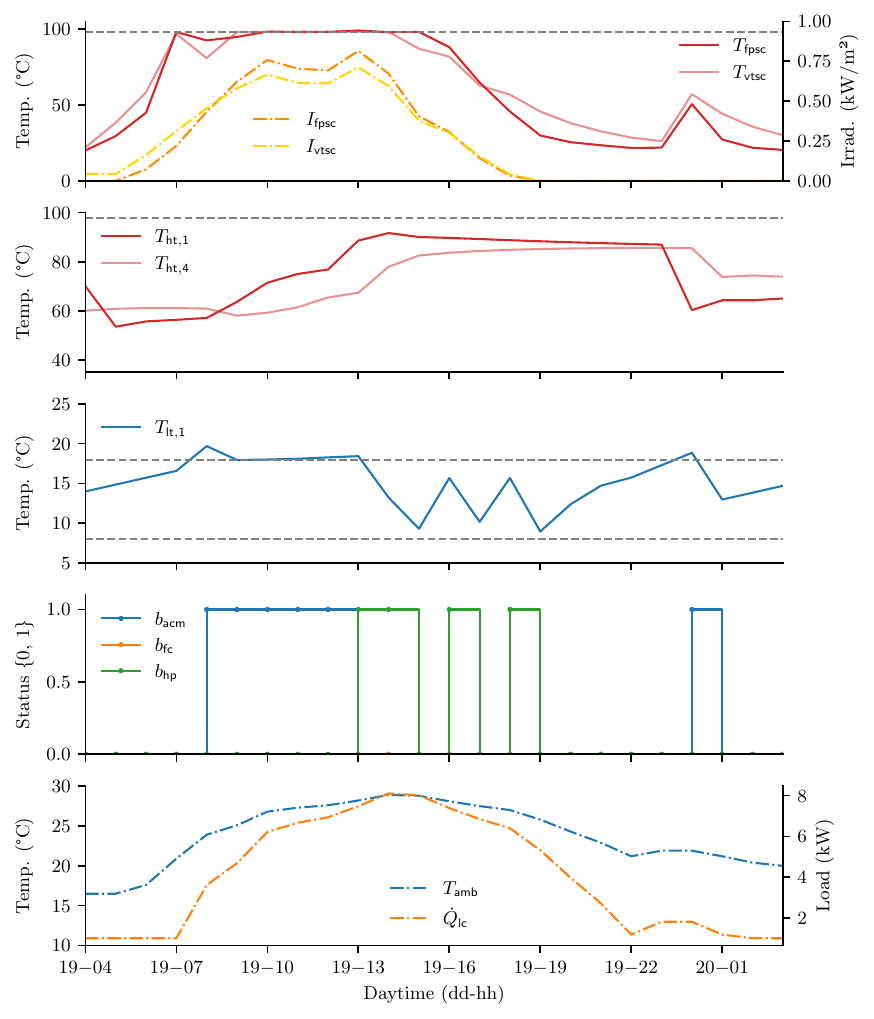}
        \caption{State and binary control trajectory obtained by solving \eqref{op: stcs ocp} via S-B-MIQP.
        The plotted solution is obtained at iteration 1 and has the lowest objective found. The bounds on state and control are the dashed gray lines. The ambient conditions are represented by dashed-dotted lines, we have solar irradiance on the flat plate and vacuum tube solar collectors, $I_\mathrm{fpsc}, I_\mathrm{vtsc}$, respectively, the ambient temperature, $T_\mathrm{amb}$, and the cooling load profile, $\dot{Q}_\mathrm{lc}$.}
        \label{fig: trajectory stcs}
\end{figure}

In Table \ref{tab: stcs s-b-miqp iteration}, we report the results of each iteration performed by S-B-MIQP.
The first row shows the objective value of the relaxed $\prob_\NLP$ and the objective of $\prob_\mathrm{BR-MIQP}$ constructed around the relaxed $\prob_\NLP$ solution, as described in Sec.~\ref{subsec: implementation details}.
Notably, for S-B-MIQP the best solution is obtained in the first iteration, by solving $J(y_1)$.
We believe this occurs because the solution of the relaxed $\prob_\NLP$ provides an informative starting point for constructing the first $\prob_\mathrm{BR-MIQP}$, whose solution already yields the best integer controls found.
During the remaining computation time, both S-B-MIQP variants attempt to improve this solution, but without success.
\andrea{
The last three columns of Table~\ref{tab: stcs s-b-miqp iteration} report the wall time of each subproblem.
The first two $\prob_\mathrm{BR-MIQP}$ are quite expensive to solve compared to $\prob_\NLP$ and the first ten $\prob_\mathrm{LB-MILP}$.
From iteration 14, also the $\prob_\mathrm{LB-MILP}$ requires longer computation times.
These computation time could be reduced implementing a single-tree strategy for the solution of $\prob_\mathrm{LB-MILP}$ \citep{Quesada1992,Abhishek2006} while adding new constraints via ``lazy constraints'' callbacks, available for instance in Gurobi.
}

\begin{table}
\setlength{\tabcolsep}{4pt}
\centering
\begin{tabular}{@{}c|rrrrrrrr@{}}
\toprule
Iter. $k$ &
$\UB$ &
$\LB$ &
$J(y_k)$ &
$V_{\mathrm{MIQP},k}$ &
$V_{\mathrm{MILP},k}$ &
NLP time &
MIQP time &
MILP time \\
\midrule
0  & $\infty$ & $-\infty$ & 1547.33$^*$ & 2134.29 & --        & 2.996 & 78.945 & --      \\
1  & $\infty$ & 1547.33   & \textbf{2288.40} & 2043.05 & -- & 5.729 & 57.598 & --      \\
2  & 2288.40 & 1547.33   & 2438.16 & --       & -10475.65 & 4.648 & --     & 0.819   \\
3  & 2288.40 & 1547.33   & 329568.02 & --    & -6604.50  & 2.315 & --     & 6.594   \\
4  & 2288.40 & 1547.33   & 92830.58 & --    & -4460.51  & 7.080 & --     & 0.974   \\
5  & 2288.40 & 1547.33   & 94530.31 & --    & -4533.46  & 5.417 & --     & 1.039   \\
6  & 2288.40 & 1547.33   & 12753.07 & --    & -2511.64  & 4.806 & --     & 1.180   \\
7  & 2288.40 & 1547.33   & 42158.56 & --    & -2156.02  & 1.647 & --     & 5.020   \\
8  & 2288.40 & 1547.33   & 12578.73 & --    & -1474.22  & 8.012 & --     & 1.008   \\
9  & 2288.40 & 1547.33   & 18825.85 & --    & -1428.67  & 11.123& --     & 8.903   \\
10 & 2288.40 & 1547.33   & 8360.52  & --    & -1412.07  & 6.128 & --     & 17.219  \\
11 & 2288.40 & 1547.33   & 5345.72  & --    & -1401.91  & 7.885 & --     & 24.223  \\
12 & 2288.40 & 1547.33   & 7135.18  & --    & -1403.42  & 6.328 & --     & 50.230  \\
13 & 2288.40 & 1547.33   & 4378.46  & --    & -1372.62  & 5.632 & --     & 52.926  \\
14 & 2288.40 & 1547.33   & 3890.60  & --    & -1348.06  & 6.702 & --     & 99.330  \\
15 & 2288.40 & 1547.33   & 3898.51  & --    & -1308.88  & 5.095 & --     & 29.526  \\
16 & 2288.40 & 1547.33   & 3956.89  & --    & -1233.70  & 9.644 & --     & 123.145 \\
17 & 2288.40 & 1547.33   & 2806.25  & --    & -1143.33  & 5.575 & --     & 46.358  \\
18 & 2288.40 & 1547.33   & 2735.66  & --    & -1112.57  & 7.614 & --     & 57.206  \\
19 & 2288.40 & 1547.33   & 3518.17  & --    & -817.74   & 4.942 & --     & 47.172  \\
20 & 2288.40 & 1547.33   & 2763.79  & --    & -645.47   & 7.038 & --     & 66.940  \\
21 & 2288.40 & 1547.33   & 3391.25  & --    & -502.84   & 7.150 & --     & 104.462 \\
22 & 2288.40 & 1547.33   & 2658.91  & --    & -266.93   & 5.481 & --     & 119.463 \\
23 & 2288.40 & 1547.33   & 2571.64  & --    & 660.16    & 4.431 & --     & 105.601 \\
24 & 2288.40 & 1547.33   & 14361.66 & --    & 660.25    & 7.390 & --     & 130.488 \\
25 & 2288.40 & 1547.33   & 11681.92 & --    & 668.18    & 11.496& --     & 135.643 \\
26 & 2288.40 & 1547.33   & 16616.70 & --    & 670.43    & 9.564 & --     & 82.265  \\
27 & 2288.40 & 1547.33   & 14503.88 & --    & 681.85    & 6.446 & --     & 82.006  \\
28$^\dagger$ & 2288.40 & 1547.33   & 14557.06 & --    & 709.98    & 5.919 & --     & 79.970  \\
\bottomrule
\multicolumn{9}{l}{$^*$relaxed solution; $^\dagger$iterations stopped due to time limit.}%
\end{tabular}
\caption{Iterations of S-B-MIQP for the solution of STCS \eqref{op: stcs ocp}}\label{tab: stcs s-b-miqp iteration}
\end{table}

To conclude, with this example we aimed to show that S-B-MIQP can be applied out-of-the-box to a large and complex \ac{MINLP} with satisfactory results.
In this case, the termination of S-B-MIQP was triggered by the time limit of 30 minutes but already in less than 2 minutes an integer feasible solution was available, which for this case coincided with the best solution found.

\section{Conclusion and outlook}\label{sec: conclusion}

We presented a novel algorithm for solving mixed-integer nonlinear programming problems.
We showed that the algorithm combines cutting planes based on generalized Benders' decomposition and outer approximation efficiently and converges to the global optimum or with a certificate of infeasibility for convex \acp{MINLP}.
We proposed an extension for treating nonconvex MINLPs employing a heuristic to modify the generated cutting planes.
The extension does not alter the results for convex MINLPs while it makes the algorithm directly applicable to nonconvex problems.
The algorithm was compared to Bonmin, Gurobi, SCIP and SHOT for a large subset of both convex and nonconvex MINLPs from the MINLPLib.
The results show that the proposed algorithm is particularly suited for nonconvex MINLPs while it closely follows the performance of SHOT for convex MINLPs.
Finally, we presented the results obtained with the proposed algorithm in two cases of optimal control for switched systems.

The open-source software package \texttt{CAMINO}\footnote[1]{\url{https://github.com/minlp-toolbox/CAMINO}}, developed to implement the proposed algorithm, is coded in Python and relies on CasADi for modelling the optimization problems and interfacing with required solvers.
Moreover, \texttt{CAMINO} includes implementations of various methods found in the literature.
\andrea{
A welcome addition is an improved AMPL-file generation from CasADi for \acp{OCP}, enabling a quick interface with any AMPL-compatible solver.
For higher efficiency, a direct interface between CasADi's expression graph and those of SHOT, SCIP, and Gurobi could be developed, though this would require more implementation effort.
}

Overall, the proposed algorithm shows promising solution quality and computation time results.
This performance might be further improved by \andrea{an efficient implementation in compiled code}, adding preprocessing routines at the MINLP level, such as bound tightening, and algorithmic strategies like the single-tree mode used in SHOT, where the solution of the master problem can resume from previous iteration without rebuilding the branch-and-bound tree.
\andrea{In the future we aim at extending the proposed algorithm to become a global solver for \emph{nonconvex} MINLPs, e.g., by using ideas from~\citep{Li2012}.}

\bibliographystyle{spbasic}      %
\bibliography{sbmiqp-bib.bib}   %

\section*{Statements and Declarations}
\subsection*{Acknowledgements}
A preliminary version of the algorithm proposed in this work has been presented in the Oberwolfach meeting ``Mixed-integer Nonlinear Optimization: A Hatchery for Modern Mathematics'' in August 2023 \cite{Oberwolfach2023}.
We thank Adrian B\"urger, Clemens Zeile, L\'eo Simpson, and Ramin Abbasi-Esfeden for inspiring discussions.

\subsection*{Funding}
Andrea Ghezzi and Moritz Diehl acknowledge funding from the European Union via ELO-X 953348, from DFG via projects 504452366 (SPP 2364), 560056112 (robust MPC), 535860958 (ALeSCo) and 525018088 (MAWERO), and from BMWK via 03EN3054B.
Wim Van Roy is supported by the Baekeland scholarship (Grant 182076), given by Atlas Copco Airpower NV, Wilrijk, Belgium, and by the Institute for the Promotion of Innovation through Science and Technology in Flanders (VLAIO), Belgium.
Sebastian Sager acknowledges funding from DFG via RTG 2297 and SPP 2331.

\subsection*{Competing Interests}
The authors have no relevant financial or non-financial interests to disclose.

\subsection*{Author Contributions}
\textbf{Andrea Ghezzi, Wim Van Roy}: Writing – original draft, Conceptualization, Visualization, Software.
\textbf{Sebastian Sager}: Writing – review \& editing, Supervision, Conceptualization.
\textbf{Moritz Diehl}: Writing – review \& editing, Supervision, Conceptualization, Resources, Project administration, Funding acquisition.

\subsection*{Code and Data Availability}
The full code is open-source, and specific references are included in this article. The packages we rely on are either open source or available for academic use. All data analyzed during this study are publicly available, URLs are included in this article.

\section*{Appendix: Tabular results of Section 4.1}

In the tables below, wall time is reported in secondsm and we use the acronym ``NaN'' when the algorithm fails to find a feasible solution within the 300-second time limit.

\begin{landscape}

\begin{scriptsize}

\begin{longtable}{@{} lrrrrrrrrrrrr @{}}
\caption{Benchmark results for convex MINLPs.}\label{tab: appendix convex problems}\\
\toprule
Problem & \multicolumn{2}{c}{Bonmin} & \multicolumn{2}{c}{Gurobi} & \multicolumn{2}{c}{SCIP} & \multicolumn{2}{c}{SHOT} & \multicolumn{2}{c}{S-B-MIQP} & \multicolumn{2}{c}{S-B-MIQP-ee} \\
 & Objective & Wall time & Objective & Wall time & Objective & Wall time & Objective & Wall time & Objective & Wall time & Objective & Wall time \\
\midrule
\endfirsthead
\toprule
Problem & \multicolumn{2}{c}{Bonmin} & \multicolumn{2}{c}{Gurobi} & \multicolumn{2}{c}{SCIP} & \multicolumn{2}{c}{SHOT} & \multicolumn{2}{c}{S-B-MIQP} & \multicolumn{2}{c}{S-B-MIQP-ee} \\
 & Objective & Wall time & Objective & Wall time & Objective & Wall time & Objective & Wall time & Objective & Wall time & Objective & Wall time \\
\midrule
\endhead
\midrule
\multicolumn{13}{r}{Continued on next page} \\
\midrule
\endfoot
\bottomrule
\endlastfoot
batch & 2.86e+05 & 0.293 & 2.86e+05 & 0.0979 & 2.86e+05 & 0.445 & 2.86e+05 & 0.111 & 3.01e+05 & 0.114 & 3.09e+05 & 0.0995 \\
batch0812 & 2.69e+06 & 1.78 & 2.7e+06 & 0.055 & 2.69e+06 & 0.269 & 2.71e+06 & 0.482 & 3.11e+06 & 0.295 & 3.11e+06 & 0.275 \\
batchdes & 1.67e+05 & 0.0487 & 1.67e+05 & 0.0148 & 1.67e+05 & 0.0278 & 1.67e+05 & 0.0676 & 1.67e+05 & 0.0273 & 1.67e+05 & 0.0322 \\
batchs101006m & 7.69e+05 & 14.5 & 7.74e+05 & 0.816 & 7.69e+05 & 6.69 & 7.77e+05 & 1.29 & 7.69e+05 & 6.31 & 7.69e+05 & 4.83 \\
batchs121208m & 1.24e+06 & 33.4 & 1.24e+06 & 1.7 & 1.24e+06 & 13.5 & 1.24e+06 & 2.39 & 1.25e+06 & 11.8 & 1.25e+06 & 5.23 \\
batchs151208m & 1.54e+06 & 170 & 1.54e+06 & 1.11 & 1.54e+06 & 14.9 & 1.55e+06 & 4.47 & 1.55e+06 & 13.5 & 1.56e+06 & 6.35 \\
batchs201210m & 2.3e+06 & 175 & 2.3e+06 & 2.07 & 2.3e+06 & 16.4 & 2.3e+06 & 6.08 & 2.3e+06 & 18.7 & 2.32e+06 & 11.9 \\
clay0203hfsg & 4.16e+04 & 9.07 & 4.16e+04 & 0.803 & 4.16e+04 & 1.23 & 4.16e+04 & 0.494 & 4.17e+04 & 3.31 & 4.21e+04 & 1.61 \\
\end{longtable}

\begin{longtable}{@{} lrrrrrrrrrrrr @{}}
\caption{Benchmark results for nonconvex MINLPs.}\label{tab: appendix nonconvex problems}\\
\toprule
Problem & \multicolumn{2}{c}{Bonmin} & \multicolumn{2}{c}{Gurobi} & \multicolumn{2}{c}{SCIP} & \multicolumn{2}{c}{SHOT} & \multicolumn{2}{c}{S-B-MIQP} & \multicolumn{2}{c}{S-B-MIQP-ee} \\
 & Objective & Wall time & Objective & Wall time & Objective & Wall time & Objective & Wall time & Objective & Wall time & Objective & Wall time \\
\midrule
\endfirsthead
\toprule
Problem & \multicolumn{2}{c}{Bonmin} & \multicolumn{2}{c}{Gurobi} & \multicolumn{2}{c}{SCIP} & \multicolumn{2}{c}{SHOT} & \multicolumn{2}{c}{S-B-MIQP} & \multicolumn{2}{c}{S-B-MIQP-ee} \\
 & Objective & Wall time & Objective & Wall time & Objective & Wall time & Objective & Wall time & Objective & Wall time & Objective & Wall time \\
\midrule
\endhead
\midrule
\multicolumn{13}{r}{Continued on next page} \\
\midrule
\endfoot
\bottomrule
\endlastfoot
4stufen & 1.17e+05 & 1.7 & 1.16e+05 & 300 & 8.32e+03 & 300 & NaN & NaN & NaN & NaN & NaN & NaN \\
autocorr\_bern20-05 & -408 & 0.0147 & -368 & 0.0509 & -337 & 0.0509 & -416 & 4.6 & -384 & 0.04 & -384 & 0.0351 \\
autocorr\_bern20-10 & -2.89e+03 & 0.186 & -2.69e+03 & 0.223 & -2.82e+03 & 0.225 & -2.94e+03 & 66 & -2.93e+03 & 0.207 & -2.93e+03 & 0.193 \\
autocorr\_bern20-15 & -5.82e+03 & 0.248 & -5.05e+03 & 0.457 & -364 & 0.487 & -5.96e+03 & 183 & -5.82e+03 & 0.438 & -5.82e+03 & 0.43 \\
autocorr\_bern25-06 & -928 & 0.039 & -592 & 0.09 & -824 & 0.0902 & -960 & 97.8 & -944 & 0.0779 & -944 & 0.0656 \\
autocorr\_bern25-13 & -8.06e+03 & 0.614 & NaN & NaN & -264 & 0.391 & -8.14e+03 & 300 & -7.99e+03 & 0.365 & -7.99e+03 & 0.348 \\
autocorr\_bern25-19 & -1.45e+04 & 1.54 & -1.1e+04 & 0.919 & -612 & 0.938 & -1.46e+04 & 300 & -1.45e+04 & 0.897 & -1.45e+04 & 0.904 \\
autocorr\_bern25-25 & -1.04e+04 & 1.8 & -9.78e+03 & 1.01 & -1.1e+03 & 1.03 & -1.05e+04 & 300 & -1.05e+04 & 0.982 & -1.05e+04 & 0.943 \\
\end{longtable}

\end{scriptsize}

\end{landscape}

\end{document}